\documentclass{amsart}

\usepackage{amssymb,amsmath,amsthm,amsfonts}
\usepackage{framed,comment,enumerate}
\usepackage[pdftex]{graphicx}
\usepackage{xcolor, framed,overpic}

\def\R {\mathbb{R}}
\def\C {\mathcal{C}}

\def\eps{\varepsilon}
\def\dist{{\rm dist}}

\newcommand{\loc}{\mathrm{loc}}
\newcommand{\Lip}{\mathrm{Lip}}

\newcommand{\pa}{\partial}
\newcommand{\mf}[1]{\mathbf{#1}}

\newtheorem{proposition}{Proposition}[section]
\newtheorem{theorem}[proposition]{Theorem}
\newtheorem*{theorem*}{Theorem}
\newtheorem{corollary}[proposition]{Corollary}
\newtheorem{lemma}[proposition]{Lemma}

\theoremstyle{definition}
\newtheorem{definition}[proposition]{Definition}

\newtheorem{remark}[proposition]{Remark}
\numberwithin{equation}{section}

\title[Phase separation: asymptotic and geometric aspects]{On phase separation in systems of coupled elliptic equations: asymptotic analysis and geometric aspects}



\author{Nicola Soave and Alessandro Zilio}

\address{
\hbox{\parbox{5.7in}{\medskip\noindent
 Nicola Soave\\
Mathematisches Institut, Justus-Liebig-Universit\"at Giessen, \\
Arndtstrasse 2, 35392 Giessen (Germany),\\[2pt]
{\em{E-mail address: }}{\tt nicola.soave@gmail.com, nicola.soave@math.uni-giessen.de.}\\[5pt]
Alessandro Zilio\\
Centre d'analyse et de math\'{e}matique sociales\\
\'{E}cole des Hautes \'{E}tudes en Sciences Sociales\\
190-198 Avenue de France, 75244, Paris CEDEX 13 (France) \\
{\em{E-mail address: }}{\tt azilio@ehess.fr, alessandro.zilio@polimi.it.}}}}

\begin{document}

\begin{abstract}
We consider a family of positive solutions to the system of $k$ components
\[
	-\Delta u_{i,\beta} = f(x, u_{i,\beta}) - \beta u_{i,\beta} \sum_{j \neq i} a_{ij} u_{j,\beta}^2 \qquad \text{in $\Omega$},
\]
where $\Omega \subset \R^N$ with $N \ge 2$. It is known that uniform bounds in $L^\infty$ of $\{\mathbf{u}_{\beta}\}$ imply convergence of the densities to a segregated configuration, as the competition parameter $\beta$ diverges to $+\infty$. In this paper 
we establish sharp quantitative point-wise estimates for the densities around the interface between different components, and we characterize the asymptotic profile of $\mf{u}_\beta$ in terms of entire solutions to the limit system
\[
	\Delta U_i = U_i \sum_{j\neq i} a_{ij} U_j^2.
\]
Moreover, we develop a uniform-in-$\beta$ regularity theory for the interfaces.
\end{abstract}

\maketitle

\section{Introduction}

The aim of this paper is to prove qualitative properties of positive solutions to competing systems with variational interaction, whose prototype is the coupled Gross-Pitaevskii equation
\[
\begin{cases}
-\Delta u_{i,\beta} +\lambda_{i,\beta} u_{i,\beta} = \mu_i u_{i,\beta}^3 -\beta u_{i,\beta} \sum_{j \neq i} a_{ij} u_{j,\beta}^2 & \text{in $\Omega$} \\
u_i >0 & \text{in $\Omega$},
\end{cases}  \quad i=1,\dots,k,
\]
in the limit of strong competition $\beta \to +\infty$. This problem naturally arises in different contexts: from the physics world, it is of interest in nonlinear optics and in the Hartree-Fock approximation for Bose-Einstein condensates with multiple hyperfine states, see e.g. \cite{AkAn, Timm}. From a mathematical point of view, it is useful in the approximation of optimal partition problems for Laplacian eigenvalues, and in the theory of harmonic maps into singular manifolds, see \cite{CaffLin, CoTeVe2002,CoTeVe2003, rtt,TaTePoin}. Several papers are devoted to the development of a common regularity theory for families of solutions associated to families of parameters $\beta \to +\infty$, to the analysis of the convergence of such families to some limit profile, and to the regularity issues for the emerging free-boundary problem, see \cite{CaffLin, CaffLin2010, ChangLinLinLin, CoTeVe2002, CoTeVe2003, NoTaTeVe, rtt, SoTaTeZi, SoZi, WeiWeth}. On the other hand, not much is known about finer qualitative properties, such as:
\begin{itemize}
\item the decay rate of convergence of the solutions,
\item the geometric structure of the solutions in a neighbourhood of the ``interface" between different components (a concept which will be conveniently defined), 
\item the geometric structure of the interface itself.
\end{itemize}
To our knowledge, the only contribution dealing with this kind of problem is \cite{BeLiWeZh}, where Berestycki et al. considered the $1$-dimensional system
\begin{equation}\label{main system}
\begin{cases}
-w_{1,\beta}'' +\lambda_{1,\beta} w_{1,\beta} = \mu_1 w_{1,\beta}^3-\beta w_{1,\beta} w_{2,\beta}^2 & \text{in $(0,1)$} \\
-w_{2,\beta}'' +\lambda_{2,\beta} w_{2,\beta} = \mu_2 w_{2,\beta}^3-\beta w_{1,\beta}^2 w_{2,\beta} & \text{in $(0,1)$} \\
w_{i,\beta},>0  \quad  \text{in $(0,1)$}, \quad w_i \in H_0^1(0,1) & i=1,2 \\
\int_0^1 w_{1,\beta}^2 = \int_0^1 w_{2,\beta}^2 = 1.
\end{cases}
\end{equation}
Under the assumption that $(\lambda_{1,\beta})$ and $(\lambda_{2,\beta})$ are bounded sequences, they proved that if $x_\beta \in \{w_{1,\beta}=w_{2,\beta}\}$ (the interface between $w_{1,\beta}$ and $w_{2,\beta}$), then there exists $C>1$ such that
\begin{equation}\label{eq: lower ext}
\frac{1}{C} \le \beta w_{1,\beta}^2(x_\beta) w_{2,\beta}^2(x_\beta) \le C \quad \forall \beta > 0;
\end{equation}
that is, any family of solutions decays, along sequences of points where $w_{1,\beta}=w_{2,\beta}$, like $\beta^{-1/4}$, see \cite[Theorem 1.1]{BeLiWeZh}. Furthermore, they showed that suitable scalings of $(w_{1,\beta},w_{2,\beta})$ in a neighbourhood of the interface converge, in $\mathcal{C}^2_{\loc}(\R)$, to an entire solution of
\begin{equation}\label{entire 1-d intro}
\begin{cases}
W_1'' = W_1 W_2^2 \\
W_2'' = W_1^2 W_2  & \text{in $\R$}\\
W_1, W_2 >0,
\end{cases} 
\end{equation}
see \cite[Theorem 1.2]{BeLiWeZh}. This means that the geometry of the solutions to \eqref{entire 1-d intro} is related to the geometry of the solutions to \eqref{main system} near the interface; and in this perspective it is remarkable that, up to scaling, translations and exchange of the components, \eqref{entire 1-d intro} has only one solution, see \cite[Theorem 1.1]{BeTeWaWe}.



The purpose of this paper is to generalize the analysis in \cite{BeLiWeZh} in higher dimension and to $k \ge 2$ components systems with general form. In order to present and motivate our study, we introduce some notation and review some known results. For simplicity, in the rest of the paper the expression ``up to a subsequence" will be understood without always being mentioned.

We consider weak solutions to
\begin{equation}\tag{$P_\beta$}\label{main system k comp}
\begin{cases}
-\Delta u_{i,\beta} = f_{i,\beta}(x,u_{i,\beta})-\beta u_{i,\beta} \sum_{j \neq i} a_{ij} u_{j,\beta}^2 & \text{in $\Omega$} \\
u_{i,\beta}>0 & \text{in $\Omega$},
\end{cases} \quad i=1,\dots,k,
\end{equation}
where $a_{ij}=a_{ji}>0$, $\beta>0$, and $\Omega$ is a domain of $\R^N$ neither necessarily bounded, nor necessarily smooth, with $N \le 4$. Since any coupling parameter $\beta a_{ij}$ is positive, with the considered sign convention the relation between any pair of densities $u_{i,\beta}$ and $u_{j,\beta}$ is of competitive type. Concerning the nonlinearities $f_{i,\beta}$, we always assume that $f_{i,\beta} \in \C^1(\Omega \times \R)$ are such that:
\begin{itemize}
\item[(F1)] $f_{i,\beta}(x,s) = O(s)$ as $s \to 0$, uniformly in $x \in \Omega$, that is there exists $C>0$ such that
    \[
        \max_{s \in [0,1]} \sup_{x \in \Omega} \left|\frac{f_{i,\beta}(x,s)}{s}\right| \leq C \qquad i=1,\dots,k;
    \]
\item[(F2)] for any sequence $\beta \to +\infty$ there exist a subsequence (still denoted $\beta$) and functions $f_i \in \mathcal{C}^1(\Omega \times \R)$ such that $f_{i,\beta} \to f_i$ in $\mathcal{C}_{\loc}(\Omega \times \R)$.
\end{itemize}

We explicitly remark that for the nonlinearity appearing in the classical Gross-Pitaevskii equation, i.e.
\[
f_{i,\beta}(x,s) = \mu_i s^3-\lambda_{i,\beta} s
\]
with $\mu_i,\lambda_{i,\beta} \in \R$, both (F1) and (F2) are satisfied provided $\{\lambda_{i,\beta}\}$ is bounded. This is exactly the assumption in \cite{BeLiWeZh}.

Let us suppose that
\begin{equation}\tag{U1}\label{boundedness}
\text{$\{\mf{u}_\beta:\beta >0\}$ is a family of solutions to \eqref{main system k comp}, uniformly bounded in $L^\infty(\Omega)$},
\end{equation} 
where we used the vector notation $\mf{u}_\beta=(u_{1,\beta},\dots,u_{k,\beta})$. Then, as we showed in \cite{SoZi}, for any compact $K \Subset \Omega$ we have that $\{\mf{u}_\beta\}$ is uniformly bounded in $\Lip(K)$. As a consequence, it is possible to infer that there exists a locally Lipschitz continuous limit $\mf{u}$ such that $\mf{u}_{\beta} \to \mf{u}$ as $\beta \to +\infty$ in $\mathcal{C}^{0,\alpha}_{\loc}(\Omega)$ (for every $0<\alpha<1$) and in $H^1_{\loc}(\Omega)$, and
\begin{equation}\label{limit system}
\begin{cases}
-\Delta u_i  = f_i(x,u_i) & \text{in $\{u_{i}>0\}$} \\
u_{i} \cdot u_{j} \equiv 0  & \text{in $\Omega$, for every $i \neq j$},
\end{cases}
\end{equation}
where $f_i$ is the limit of the considered sequence $\{f_{i,\beta}\}$ (see \cite{NoTaTeVe,SoTaTeZi,Wa}).

In the present paper we always assume that (F1), (F2) and \eqref{boundedness} are satisfied, and therefore we will not explicitly recall them in all our statements. Moreover, from now on we shall always focus on a particular converging subsequence, and on the corresponding limit profile, without changing the notation for the sake of simplicity. 

Since the limit $\mf{u}$ is segregated, it is natural to define the \emph{nodal set}, or \emph{free-boundary}, as $\Gamma:= \{ u_{i}=0 \ \text{for every $i$}\}$. The properties of the free-boundary were studied in \cite{tt} (see also \cite{CaffLin}). As limit of strongly competing system, \cite[Theorem 8.1]{tt} establishes that $\mf{u}$ belongs to a class of segregated vector valued functions, called $\mathcal{G}(\Omega)$ (see Definition 1.2 in \cite{tt}), and hence the nodal set $\Gamma$ has the following properties: it has Hausdorff dimension $N-1$, and it is decomposed into two parts $\mathcal{R}$ and $\Sigma$. The set $\mathcal{R}$, called \emph{regular part}, is relatively open in $\Gamma$ and is the union of hyper-surfaces of class $\mathcal{C}^{1,\alpha}$ for every $0<\alpha<1$. The set $\Sigma=\Gamma \setminus \mathcal{R}$, the \emph{singular part}, is relatively closed in $\Gamma$ and has Hausdorff dimension at most $N-2$. By means of the \emph{Almgren's frequency function}
\begin{equation}\label{def: Almgren intro}
N(\mf{u},x_0,r):= \frac{r \int_{B_r(x_0)}  \sum_{i=1}^k \left( |\nabla u_i|^2 - f_i(x,u_i) u_i   \right)    }{ \int_{\pa B_r(x_0)} \sum_{i=1}^k u_i^2},
\end{equation}
regular and singular part are defined by
\[
\mathcal{R} := \left\{ x \in \Gamma: N(\mf{u},x_0,0^+)=1 \right\}, \quad \quad
\Sigma:= \left\{ x \in \Gamma: N(\mf{u},x_0,0^+)>1 \right\}.
\]
Combining the results in \cite{tt} with those in Section 10 of \cite{DaWaZh}, it is possible to deduce also that every point $x_0 \in \mathcal{R}$ has multiplicity exactly equal to $2$, that is 
\[
\# \left\{i=1,\dots,k: \text{ for every $r>0$ it results $B_r(x_0) \cap \{u_i>0\} \neq \emptyset$}\right\}=2.
\]
This prevents in particular the occurrence of self-segregation.

\subsection{Main results: asymptotic estimates}

For a large part of the paper we will be interested in studying the decay rate of the sequence $\{\mf{u}_\beta\}$ as $\beta \to +\infty$ in a neighbourhood of points of the free-boundary $\Gamma=\{\mf{u}=\mf{0}\}$. As already recalled, the only results available in this context are those contained in \cite[Theorem 1.1]{BeLiWeZh}. We mention that decay estimates are not only relevant for themselves, but are useful since they suggest the correct asymptotic behaviour in some approximated optimal partition problems, finally leading to powerful monotonicity formulae for competing systems, see \cite[Theorem 1.6]{BeLiWeZh}, \cite[Theorem 5.6]{BeTeWaWe} and \cite[Lemma 4.2]{Wa}.

In what follows we discuss the generalization of the analysis in \cite{BeLiWeZh} in higher dimension. Already in the plane, the situation is much more involved with respect to the $1$-dimensional problem: first, due to the richer structure of the free-boundary $\Gamma$.
Second, due to the fact that we deal with more than $2$ components, so that, as we shall see, we have to distinguish the dominating functions (which will have a suitable decay) from the other ones (which will decay much faster). Finally due to fact that the Hamiltonian structure of the problem, one of the key tools used in \cite{BeLiWeZh} for the proof of \eqref{eq: lower ext} in dimension $N=1$, is lost for general nonlinearities $f_{i,\beta}$, and in any case is much less powerful in subsets of $\R^N$ with $N \ge 2$ than in $\R$ (we point out that for general $f_{i,\beta}$ the forthcoming results are new results also in dimension $N=1$). 

In order to overcome these difficulties, we develop a new approach based upon monotonicity formulae and tools from geometric measure theory.

The first of our results is a consequence of the uniform Lipschitz boundedness of $\{\mf{u}_\beta\}$ in compacts of $\Omega$, see \cite{SoZi}, and extends the upper estimate in \eqref{eq: lower ext} to the present setting.

\begin{theorem}\label{thm: global upper estimate}
For every compact set $K \Subset \Omega$ there exists $C >0$ such that
\[
\beta u_{i,\beta}^2 u_{j,\beta}^2 \le C \quad \text{in $K$, for every $i \neq j$}.
\]
\end{theorem}

Notice that, by the lower estimate in \eqref{eq: lower ext}, the result is optimal in general.

In order to derive finer properties, we introduce the concept of \emph{interface} of $\mf{u}_\beta$. 

\begin{definition}\label{def interface}
We define the \emph{interface} of $\mf{u}_\beta$ as 
\[
\Gamma_\beta:= \left\{ x \in \Omega \left| \begin{array}{l}\text{$u_{i,\beta}(x) = u_{j,\beta}(x)$ for some $i \neq j$} \\
\text{and $u_{i,\beta}(x) \ge u_{l,\beta}(x)$ for all the other indices $l$}\end{array}\right. \right\}
\]
\end{definition}

Roughly speaking, a point $x$ is on the interface of $\mf{u}_\beta$ if at least two components coincide in $x$, and the remaining ones are smaller. Notice that, if the number of components is $k=2$, then the interface is naturally defined as 
\[
\Gamma_\beta:= \{ u_{1,\beta} = u_{2,\beta}\}.
\]

As we shall see, the interface plays the role of the free boundary $\Gamma=\{\mf{u}=\mf{0}\}$ for the $\beta$-problem \eqref{main system k comp}. A simple intuitive reason for this is that any converging sequence of points in $\Gamma_\beta$ necessarily converges to a limit in $\Gamma$. Moreover, if $x \in \Gamma$, then there exists a sequence of points in $\Gamma_\beta$ approaching $x$.

\begin{proposition}\label{prop: interfaces are good approximation}
If $x_\beta \in \Gamma_\beta$ and $x_\beta \to x_0 \in \Omega$ as $\beta \to +\infty$, then $x_0 \in \Gamma$. Moreover, if $\mf{u} \not \equiv \mf{0}$, then for any $x_0 \in \Gamma$ there exists $x_\beta \in \Gamma_\beta$ such that $x_\beta \to x_0$.
\end{proposition}

%

In what follows we consider the problem of estimating the rate of convergence of $\mf{u}_\beta$ in sequences of points on the interfaces $\Gamma_\beta$. By Theorem \ref{thm: global upper estimate}, if $x_\beta \in \Gamma_\beta$, then
\begin{equation}\label{eq: general decay}
u_{i,\beta}(x_\beta) \le \frac{C}{\beta^{1/4}} \qquad \text{for every $i=1,\dots,k$}.
\end{equation}

This estimate holds for all the components $u_{i,\beta}$. On the other hand, since on the interface we have two (or more) components dominating over the others, it is natural to expect that for the remaining ones the rate of convergence to $0$ is faster. We can prove this assuming that $x_\beta \to x_0 \in \mathcal{R}$, the regular part of $\Gamma$; recall that in this case $x_0$ has multiplicity $2$.

\begin{theorem}\label{corol: decay components vanishing}
Let  $x_0 \in \mathcal{R}$. Let $i_1$ and $i_2$ be the only two indices such that $u_{i_1}, u_{i_2} \not \equiv 0$ in a neighbourhood of $x_0$. There exist a radius $R>0$ and a constant $C>0$ independent of $\beta \gg 1$ such that:
\begin{itemize}
\item $B_R(x_0) \cap \Gamma_\beta = \{ u_{i_1, \beta} = u_{i_2, \beta} \} \cap B_R(x_0)$, and moreover $B_R(x_0) \setminus \Gamma_\beta$ is constituted exactly by two connected components which are $\{ u_{i_1, \beta} > u_{i_2, \beta} \} \cap B_R(x_0)$ and $\{ u_{i_1, \beta} < u_{i_2, \beta} \} \cap B_R(x_0)$;
	\item for any $j \neq i_1, i_2$, the density $u_{j,\beta}$ decays exponentially, in the sense that there exists $C_1,C_2>0$ independent of $\beta$ such that
	\[
\sup_{B_R(x_0)} u_{j,\beta} \le C_1 e^{-C_1 \beta^{C_2}}. 
\]
	\item in $B_R(x_0)$ the system reduces to 
	\[
		\begin{cases}
			-\Delta u_{i_1,\beta} =f_{i_1,\beta}(x,u_{i_1,\beta}) - \beta u_{i_1,\beta} u_{i_2,\beta}^2 - u_{i_1,\beta}o_\beta(1) \\
			-\Delta u_{i_2,\beta} =f_{i_2,\beta}(x,u_{i_2,\beta}) - \beta u_{i_2,\beta} u_{i_1,\beta}^2 - u_{i_2,\beta}o_\beta(1)
		\end{cases}
	\]
	where $o_\beta(1)$ is a exponentially small perturbation in the $L^{\infty}$-norm.
\end{itemize}
\end{theorem}

%

The theorem establishes that the components that converge to zero in a neighbourhood of $x_0$ decay much faster (indeed exponentially in $\beta$) than those who survive in the limit. This, although naturally expected, is far from being trivial, and is new also in dimension $N=1$. Moreover, an important consequence of the first point is that in a neighbourhood of any point $x_0 \in \mathcal{R}$ the interfaces $\Gamma_\beta$ do not self-intersect, and separates $B_R(x_0)$ in exactly two connected components.

%
%
We now turn to the problem of extending the lower bound in \eqref{eq: lower ext} to higher dimension. It is interesting that such estimate does not always hold; this is related to the fact that, while in $\R$ the free-boundary is made of single points and is purely regular, in $\R^N$ with $N \ge 2$ the singular part $\Sigma$ appears, and it turns out that therein the decay of the solutions is faster. In order to prove this, we suppose that:
\begin{equation}\tag{U2}\label{nontrivial}
\text{the limit profile of the sequence $\{\mf{u}_\beta\}$ is $\mf{u} \not \equiv \mf{0}$}.
\end{equation}
When compared with the setting considered in \cite{BeLiWeZh}, equation \eqref{main system}, \eqref{nontrivial} reduces to the normalization condition on the $L^2$-mass of the components (our assumption is in fact weaker). 

\begin{theorem}\label{thm: decay higher multiplicity}
Under \eqref{nontrivial}, if $x_\beta \in \Gamma_\beta$ and $x_\beta \to x_0 \in \Sigma$ as $\beta \to +\infty$, then 
\[
\limsup_{\beta \to +\infty} \beta^{1/4} \left(\sum_{i=1}^k u_{i,\beta}(x_\beta)\right) = 0.
\]
\end{theorem}

The previous result leaves open the possibility that the lower estimate  \eqref{eq: lower ext} still holds for sequences in $\Gamma_\beta$  converging to the regular part of the free-boundary. We believe that this is the case, but for the moment we can only prove a sub-optimal result.

\begin{theorem}\label{thm: lowe estimate}
Under \eqref{nontrivial}, let $x_\beta \in \Gamma_\beta$, and suppose that $x_\beta \to x_0 \in \mathcal{R} \subset \Gamma$ as $\beta \to +\infty$. Let $i_1$ and $i_2$ be the only two indices such that $u_{i_1}, u_{i_2} \not \equiv 0$ in a neighbourhood of $x_0$. 
Then, for every $\eps>0$ there exists $C_\eps>0$ such that
\[
\beta^{1/4+\eps}u_{i_1,\beta}(x_\beta) = \beta^{1/4+\eps}u_{i_2,\beta}(x_\beta)\ge C_\eps.
\]
\end{theorem}
We conjecture that the previous estimate holds replacing the exponent $1/4+\eps$ with $1/4$. Theorem \ref{thm: lowe estimate} is actually a corollary of a more general statement. We recall the definition of the Almgren quotient, equation \eqref{def: Almgren intro}.

\begin{theorem}\label{thm: lower general}
Under assumption \eqref{nontrivial}, let $x_\beta \in \Gamma_\beta$ such that $x_\beta \to x_0 \in \Gamma$. Let $N(\mf{u},x_0,0^+)=D$. Then for any $\eps>0$ there exists $C_\eps>0$ such that
\[
	\liminf_{\beta \to +\infty} \beta^{(D+\eps)/(2+2D)} \left(\sum_{i=1}^k u_{i,\beta}(x) \right)\ge C_\eps.
\]
\end{theorem}

%
%
%

The last asymptotic estimate we present regards the quantification of the improvement of the decay arround the singular part of the free-boundary, under additional assumptions. We suppose that \begin{equation}\tag{A}\label{a_ij=1}
\text{$a_{ij} = 1$ for every $i \neq j$}, 
\end{equation}
and introduce the following notion:

\begin{definition}\label{def singular interface}
We define the \emph{singular part of the interface $\Gamma_\beta$} as 
\[
\Sigma_\beta:=\Gamma_\beta \setminus \left\{ x \in \Gamma_\beta\left| \begin{array}{l}
\text{there exist exactly two indices $i_1 \neq i_2$ such that} \\
\text{$u_{i_1,\beta}(x) = u_{i_2,\beta}(x) > u_{j,\beta}(x)$ for all $j \neq i_1,i_2$}, \\
\text{and $\nabla(u_{i_1,\beta} -u_{i_2,\beta})(x) \neq 0$}
\end{array} \right. \right\}.
\]
\end{definition}

The definition is inspired by classical contributions regarding the singular set of solutions of elliptic equations, see for instance \cite{HanHardtLin} and the references therein. Actually, thanks to (F1) the main results in \cite{HanHardtLin} are applicable for any $\beta$ fixed, and hence the closed set $\Gamma_\beta$ can be decomposed in $(\Gamma_\beta \setminus \Sigma_\beta) \cup \Sigma_\beta$, where $\Gamma_\beta \setminus \Sigma_\beta$ is relatively open in $\Gamma_\beta$ and is the collection of $\mathcal{C}^{1,\alpha}$ hyper-surfaces, while $\Sigma_\beta$ is relatively closed and has Hausdorff dimension at most $N-2$. In other words, the same decomposition holding for $\Gamma$ holds also for $\Gamma_\beta$.

\begin{theorem}\label{cor: improved decay singular sequence}
Under assumptions \eqref{nontrivial} and \eqref{a_ij=1}, let $x_\beta \in \Sigma_\beta$ for every $\beta$, $x_\beta \to x_0$ as $\beta \to +\infty$. Then $x_0 \in \Sigma$, and for every $\eps>0$ there exists $C_\eps>0$ such that 
\[
\limsup_{\beta \to +\infty} \beta^{3/10-\eps} \left( \sum_{i=1}^k u_{i,\beta}(x_\beta)\right) \le C_\eps.
\]
\end{theorem}

Condition $x_\beta \in \Sigma_\beta$ means that we can reach the singular part of the free-boundary through a sequence of points on the singular part of $\Gamma_\beta$ (in general the existence is not guaranteed).

\begin{remark}
In the previous discussion we believe that assumption \eqref{a_ij=1} is not really necessary, and that the last result hold also for general symmetric matrices $(a_{ij})_{i,j}$. Nevertheless, in the proof of Theorem \ref{cor: improved decay singular sequence} we shall make use of several intermediate results proved in \cite{BeTeWaWe, SoTe}, where a system with $a_{ij}=1$ is considered. For this reason, we prefer to assume \eqref{a_ij=1}. 
\end{remark}

In what follows we briefly describe the strategy of the proofs of the previous results. While Theorem \ref{thm: global upper estimate} rests essentially only on the Lipschitz boundedness of $\{\mf{u}_\beta\}$ in compacts of $\Omega$, the other decay estimates are much more involved and require several intermediate propositions of independent interest.

\subsection{Main results: normalization and blow-up}

The following is a crucial intermediate step in the proofs of Theorems \ref{thm: decay higher multiplicity}-\ref{cor: improved decay singular sequence}:
\begin{theorem}\label{thm: blow-up}
Under assumption \eqref{nontrivial}, let $x_\beta \in \Gamma_\beta$, and suppose that $x_\beta \to x_0 \in \Gamma$ as $\beta \to +\infty$. There exists a sequence of radii $r_\beta>0$, $r_\beta \to 0$ as $\beta \to +\infty$, such that the scaled sequence
\[
\mf{v}_\beta(x):=\frac{\mf{u}_{\beta}(x_\beta + r_\beta x)}{H(\mf{u}_\beta,x_\beta,r_\beta)^{1/2}}, \quad \text{where} \quad H(\mf{u}_\beta,x_\beta,r_\beta) := \frac{1}{r_\beta^{N-1}} \int_{\pa B_{r_\beta}} \sum_{i=1}^k u_{i,\beta}^2,
\]   
is convergent in $\mathcal{C}^2_{\loc}(\R^N)$ to a limit $\mf{V}$, solution to
\begin{equation}\label{entire system a_{ij}}
\begin{cases}
\Delta V_i = \sum_{i \neq j} a_{ij} V_i V_j^2  \\
V_i \ge 0
\end{cases} \quad \text{in $\R^N$}.
\end{equation}
The profile $\mf{V}$ has at least two non-trivial components and at most polynomial growth, in the sense that 
\[
V_1(x) + \dots+ V_k(x) \le C(1+|x|^d) \qquad \forall x \in \R^N
\]
for some $C,d \ge 1$.
\end{theorem}

Hence, for any dimension $N \ge 1$, the geometry of the solutions with polynomial growth of \eqref{entire system a_{ij}} is responsible for the geometry of $\mf{u}_\beta$ near the interface $\Gamma_\beta$, at least for $\beta$ sufficiently large (cf. \cite[Theorem 1.2]{BeLiWeZh}).

In this perspective, we can completely characterize the solution $\mf{V}$, and hence the geometry of $\{\mf{u}_\beta\}$, around the regular part of the free boundary.

\begin{corollary}\label{thm: classification limits regular part}
Under the assumptions of Theorem \ref{thm: blow-up}, let $x_0 \in \mathcal{R}$. Then $\mf{V}$ has only two non-trivial components, say $V_1$ and $V_2$; $(V_1,V_2)$ has linear growth, and is the unique $1$-dimensional solution of
\begin{equation}\label{system 2}
\begin{cases}
\Delta V_1 = a_{12} V_1 V_2^2 \\ 
\Delta V_2 = a_{12} V_1^2 V_2\\
V_1,V_2>0 
\end{cases} \quad \text{in $\R^N$.}
\end{equation}
\end{corollary}

Here and in what follows we write that a function is $1$-dimensional if, up to a rotation, it depends only on one variable. We postpone a detailed review of the known results about \eqref{entire system a_{ij}} to Section \ref{sec: prel}. For the moment, we anticipate that solutions of \eqref{entire system a_{ij}} having linear growth are classified: up to rigid motions and suitable scaling, there exists a unique $1$-dimensional solution \cite{SoTe,Wa,Wa2}. Therefore, the theorem establishes that, along sequences of points converging to the regular part of $\Gamma$, suitable scaling of the original solutions approaches a uniquely determined archetype profile in $\mathcal{C}^2$-sense. 

If $x_0 \in\Sigma$, the singular part of the free-boundary, then the picture is more involved and a complete classification of the admissible limits solving \eqref{entire system a_{ij}} seems out of reach. Indeed, in such case the emerging profile $\mf{V}$ has not linear growth, and \eqref{entire system a_{ij}} has infinitely many distinct solutions superlinear solutions \cite{BeTeWaWe,SoZi2,SoZi1}.
In any case, under additional assumptions we can still say something on the emerging limit profile. Recall that $\Sigma_\beta$ has been defined in Definition \ref{def singular interface}.


%
%

\begin{corollary}\label{thm: non-simple blow-up}
Under assumptions \eqref{nontrivial} and \eqref{a_ij=1}, let $x_\beta \in \Sigma_\beta$ for every $\beta$. Then $x_0 \in \Sigma$, and the limit profile $\mf{V}$ obtained in Theorem \ref{thm: blow-up} is not $1$-dimensional.
\end{corollary}

The relation between Theorem \ref{thm: blow-up} and the proofs of Theorems \ref{thm: decay higher multiplicity}-\ref{cor: improved decay singular sequence} can be summarized by the following simple idea:
\begin{itemize}
\item firstly, we can deduce properties of the emerging limit $\mf{V}$, imposing different assumptions on $x_\beta$;
\item secondly, we can use the properties of $\mf{V}$ in order to prove the desired decay estimates.
\end{itemize}
For instance Corollary \ref{thm: non-simple blow-up} will be the base point in the derivation of Theorem \ref{cor: improved decay singular sequence}.


%
%

\subsection{Main results: uniform regularity of the interfaces and its consequences}

We now present our analysis concerning uniform regularity properties for the interfaces $\Gamma_\beta$ away from its singular set $\Sigma_\beta$.
Notice that, by definition and by the regularity of $\mf{u}_\beta$ for $\beta$ fixed, the sets $\Gamma_{\beta}$ are closed subsets. Moreover, $\Sigma_\beta$ is a relatively closed subset of $\Gamma_\beta$.
It is now the time to introduce a convenient notion of ``regular part" of $\Gamma_\beta$.

\begin{definition}\label{def: regular interface}
For $\rho > 0$ fixed, we let
\[
	\mathcal{R}_{\beta}(\rho) = \left\{ x \in \Gamma_{\beta} : N_\beta(\mf{u}_\beta, x, \rho) < 1+ \frac14 \right\}.
\]
\end{definition}

\begin{figure}[h]
	\centering
	\begin{overpic}[width=0.6\textwidth]{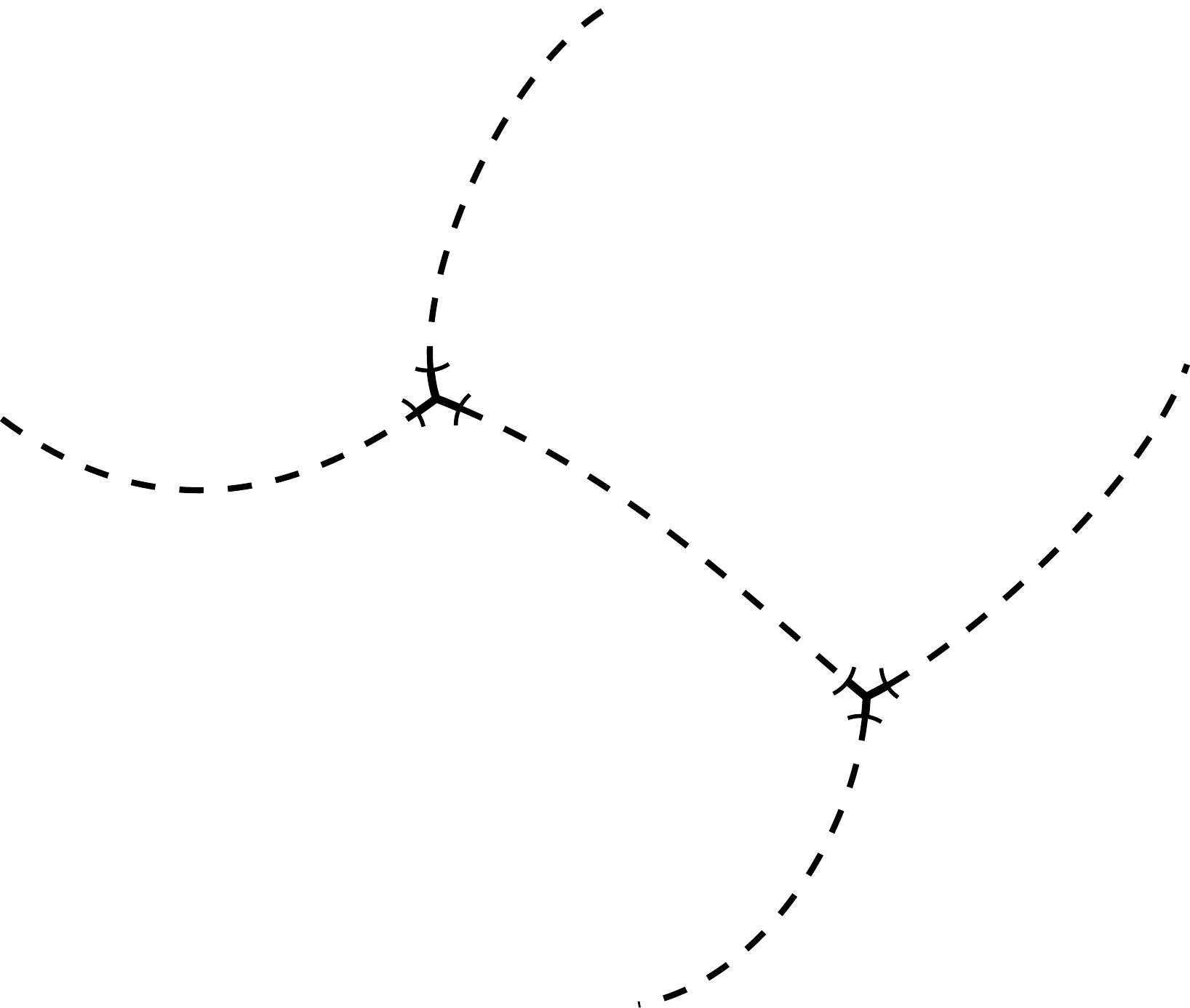}
		\put(30,20) {$u_{i,\beta} > u_{j,\beta}$}
		\put(60,40) {$u_{i,\beta}=u_{j,\beta}$}
		\put(76,21) {$\Sigma_\beta$}
		\put(40,37) {$\mathcal{R}_\beta(\rho)$}
	\end{overpic}
	\caption{A sketch of the interface $\Gamma_\beta$ for $\beta$ fixed: the dashed line represents the regular part of the free boundary $\mathcal{R}_\beta(\rho)$, while the corner points belong to the singular part $\Sigma_\beta$. As it will be proved, the singular part $\Sigma_\beta$ is detached from $\mathcal{R}_\beta(\rho)$.}
\end{figure}
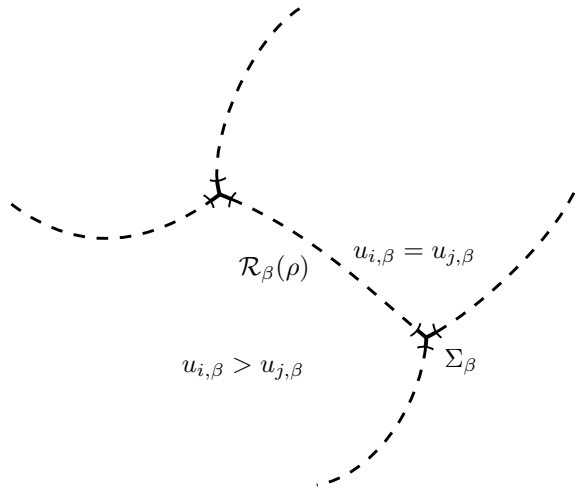

As we shall see in Lemma \ref{lem: basic prop regular part}, by taking the parameter $\rho$ sufficiently small, the sets $\mathcal{R}_{\beta}(\rho)$ is a subset of $\Gamma_\beta \setminus \Sigma_\beta$ and
is detached, uniformly in $\beta$, from the singular part $\Sigma$ of the limit free-boundary $\Gamma=\{\mf{u}=\mf{0}\}$ (and thus is also uniformly detached from $\Sigma_\beta$). Our main result states that for any fixed $\rho > 0$, the sets $\mathcal{R}_{\beta}(\rho)$ enjoy a \emph{uniform vanishing Reifenberg flatness condition}. Specifically, we have:

\begin{theorem}\label{thm: reifenberg flat uniform}
Let $K \Subset \Omega$ be a compact set, let $\rho>0$, and let us assume that \eqref{nontrivial} holds. For any $\delta > 0$ there exists $R > 0$ such that for any $x_\beta \in \mathcal{R}_{\beta}(\rho) \cap K$ and $0 < r < R$ there exists a hyper-plane $H_{x_\beta,r} \subset \R^N$ containing $x_\beta$ such that
\[
	\dist_{\mathcal{H}}(\mathcal{R}_{\beta}(\rho) \cap B_r(x_\beta), H_{x_\beta,r} \cap B_r(x_\beta)) \leq \delta r.
\]
\end{theorem}
Here and in what follows $\dist_{\mathcal{H}}$ denotes the Hausdorff distance, defined by
\[
\dist_{\mathcal{H}}(A,B) := \sup\left\{\sup_{a \in A} \dist(a,B), \sup_{b \in B} \dist(b, A) \right\}.
\]
We emphasize that in the previous theorem, the radius $R$ depends on $\rho$ and $\delta$, but not on $\beta$: this is what we mean writing that the condition holds uniformly.

The uniform vanishing Reifenberg flatness condition has several consequences:
first, it implies a uniform-in-$\beta$ local separation property of $\Gamma_\beta$ in a neighbourhood of any point of $\mathcal{R}_\beta(\rho)$. In turn, recalling also Proposition \ref{prop: interfaces are good approximation}, is the key in the proof of Theorem \ref{corol: decay components vanishing}.

At the moment we do not know if the vanishing Reifenberg flatness condition is the optimal property which holds, uniformly in $\beta$, for a subset of $\Gamma_\beta \setminus \Sigma_\beta$. In order to understand if Theorem \ref{thm: reifenberg flat uniform} is really satisfying or not, let us focus for simplicity on a $2$ components system, so that $\Gamma_\beta = \{ u_{1,\beta}-u_{2,\beta} =0\}$,
and let $x_0$ be a regular point of the limit free boundary $\Gamma=\{\mf{u}=\mf{0}\}$. Recalling the decompositions of $\Gamma$ and $\Gamma_\beta$ in regular and singular part, and also the first point in Theorem \ref{corol: decay components vanishing}, we know that for $R>0$ small enough $\{\Gamma_\beta \cap B_R(x_0): \beta\} \cup \Gamma \cap B_R(x_0)$ is a family of $\mathcal{C}^{1,\alpha}$-hypersurfaces. It is natural to wonder if this family is uniformly of class $\mathcal{C}^{1,\alpha}$, that is, if any $\Gamma_\beta$ is locally the graph of a function $\phi_\beta$, with $\{\phi_{\beta}\}$ bounded in $\mathcal{C}^{1,\alpha}$. This would imply the uniform Reifenberg flatness, being a much stronger results. A natural attempt in order to prove uniform $\mathcal{C}^{1,\alpha}$ regularity consists in trying to show that $\{u_{1,\beta}-u_{2,\beta}\}$ is uniformly bounded in $\mathcal{C}^{1,\alpha}(B_R(x_0))$ (we recall that by the reflection law in \cite{tt}, even though the limit function $\mf{u}$ is not regular, the difference $u_1-u_2$ is of class $\mathcal{C}^1$ in a neighbourhood of any point in the regular part of $\Gamma$). With the $\mathcal{C}^{1,\alpha}$-boundedness of $\{u_{1,\beta}-u_{2,\beta}\}$ and other considerations, one could prove the uniform $\mathcal{C}^{1,\alpha}$ regularity of $\Gamma_\beta \cap B_R(x_0)$, thus a natural question is now: is it true that $\{u_{1,\beta}-u_{2,\beta}\}$ is uniformly bounded in $\mathcal{C}^{1,\alpha}$ in $B_R(x_0)$? 

\begin{proposition}\label{prop: non C^1}
If $x_\beta \in \Gamma_\beta$ is such that $x_\beta \to x_0 \in \mathcal{R}$, then in general
\[
\lim_{\beta \to +\infty} \nabla (u_{1,\beta}-u_{2,\beta})(x_\beta) \neq \nabla (u_1-u_2)(x_0).
\]
In particular, in this case $\{u_{1,\beta}-u_{2,\beta}\}$ cannot be bounded in $\mathcal{C}^{1,\alpha}$.
\end{proposition}

For this reason, we think that the uniform Reifenberg flatness can be already considered as a relevant result.

\begin{remark}
Thanks to \cite[Section 8]{tt}, it is known that limits of the strongly competing system \eqref{main system k comp} share a number of properties with limits of the Lotka-Volterra system
\begin{equation}\label{LV}
-\Delta u_{i,\beta} = f_{i,\beta}(x,u_{i,\beta})-\beta u_{i,\beta} \sum_{j \neq i} a_{ij} u_{j,\beta} \qquad \text{in $\Omega$}. 
\end{equation}
It is then remarkable to observe that, while by the nature of the interaction $\{u_{1,\beta}-u_{2,\beta}\}$ is uniformly bounded in $\mathcal{C}^{1,\alpha}$ if $\{\mf{u}_\beta\}$ is a family of solutions to \eqref{LV}, this is not the case for \eqref{main system k comp}.
\end{remark}

\begin{remark}
At a first glance the reader could think that $\Gamma_\beta \setminus \Sigma_\beta$ would have been a more natural notion of regular part of $\Gamma_\beta$. But we point out that we cannot expect any uniform-in-$\beta$ regularity property for $\Gamma_\beta \setminus \Sigma_\beta$, since this relatively open subset of $\Gamma_\beta$ naturally approaches the singular part $\Sigma_\beta$ (and thus $\Sigma$). This is why we introduced $\mathcal{R}_\beta(\rho)$.   
\end{remark}

\subsection*{Structure of the paper and some notation} The second section is devoted to some preliminaries on monotonicity formulae for solutions to \eqref{main system k comp} and their limits, most of which are already known, and to the collection of some useful results regarding entire solutions of system \eqref{entire system a_{ij}}.  
Theorem \ref{thm: global upper estimate} and Proposition \ref{prop: interfaces are good approximation} are proved in Section \ref{sec: decay 1}. Theorems \ref{thm: decay higher multiplicity}-\ref{cor: improved decay singular sequence} are the object of Section \ref{sec: decay 2}, where we also prove Theorem \ref{thm: blow-up} and its corollaries. The uniform Reifeberg flatness condition and its consequences, among which Theorem \ref{corol: decay components vanishing}, are addressed in Section \ref{sec: Reif}.

With the exception of the proof of Theorem \ref{thm: global upper estimate}, we will consider for the sake of simplicity the system with $f_{i,\beta} \equiv0$, that is
\begin{equation}\label{system simplified}
\begin{cases}
\Delta u_i= \beta u_i  \sum_{j \neq i} a_{ij} u_j^2  &\text{ in $\Omega$}\\
u_i > 0 &\text{ in $\Omega$},
\end{cases}
\end{equation}
with $a_{ij}=a_{ji}>0$ and $\beta>0$. All the results that we present hold for the complete system \eqref{main system k comp}, as stated in the introduction. The proofs differ mainly for technical details, related to the fact that we shall use several monotonicity formulae, which in presence of $f_{i,\beta} \not \equiv 0$ become almost-monotonicity formulae, and hence in most of the forthcoming estimates exponential remainder terms appear. The point is that, thanks to (F1) and \eqref{boundedness}, such terms can be conveniently controlled. The interested reader can fill the details combining the approach here with that in \cite{SoZi}, where all the results are proved in full generality, and where we had to deal with the same technical complications, see also the remarks in the next sections for further details. We chose to focus on system \eqref{system simplified} with the aim of making our ideas more transparent, and the proofs technically simpler.

In this paper we adopt a notation which is mainly standard. We mention that we denote by $B_r(x)$ the ball of center $x$ and radius $r$, writing simply $B_r$ in the frequent case $x=0$. We recall that we often omit the expression ``up to a subsequence". Finally, $C$ will always denote a positive constant independent of $\beta$, whose exact value can be different from line to line.

\section{Preliminaries}\label{sec: prel}

\subsection{Monotonicity formulae for solutions to competing systems}
We collect some known and new results concerning monotonicity formulae for solutions of \eqref{system simplified}, for which we refer to \cite[Section 3.1]{SoZi} (see also \cite{BeTeWaWe,CaffLin,NoTaTeVe} for similar results).

For $x_0 \in \Omega$ and $r>0$ such that $B_r(x_0) \Subset \Omega$, we define
\begin{equation}\label{def N regular}
\begin{split}
  \bullet \quad & H(\mf{u},x_0,r):= \frac{1}{r^{N-1}} \int_{\partial B_r(x_0)} \sum_{i=1}^k u_i^2\\
  \bullet \quad & E(\mf{u},x_0,r):= \frac{1}{r^{N-2}} \int_{B_r(x_0)} \sum_{i=1}^k |\nabla u_i|^2+ 2\beta\sum_{1\le i<j\le k} a_{ij} u_i^2 u_j^2\\
  \bullet \quad & N(\mf{u},x_0,r):= \frac{E(\mf{u},x_0,r)}{H(\mf{u},x_0,r)} \qquad (\text{Almgren frequency function}).
\end{split}
\end{equation}


\begin{proposition}\label{prop: almgren}
In the previous setting, for $N \le 4$ the function $r \mapsto N(\mf{u},x_0,r)$ is monotone non-decreasing.
Moreover,
\begin{equation}\label{der H}
\frac{d}{dr} \log H(\mf{u},x_0,r) = \frac{2}{r} N(\mf{u},x_0,r) \ge 0.
\end{equation}
\end{proposition}

As a consequence of the monotonicity of the Almgren frequency function, we have the following.

\begin{proposition}\label{prop: monot e+h}
Let $\mf{u}$ be a solution of \eqref{system simplified}, and for some $x_0 \in \Omega$ and $\tilde r>0$, let $\gamma:= N(\mf{u},x_0,\tilde r)$. Then
\[
r \mapsto \frac{E(\mf{u},x_0,r)+H(\mf{u},x_0,r)}{r^{2\gamma}} \quad \text{is non-decreasing for $r>\tilde r$}.
\]
\end{proposition}
\begin{proof}
At first, integrating \eqref{der H} in $(\tilde r,r)$, we deduce that
\[
r \mapsto \frac{H(\mf{u},x_0,r)}{r^{2\gamma}} \quad \text{is non-decreasing for $r>\tilde r$}.
\]
Therefore, using also the monotonicity of $N(\mf{u},x_0,\cdot)$, it results
\begin{align*}
\frac{d}{dr} \log \left( \frac{E(\mf{u},x_0,r)+H(\mf{u},x_0,r)}{r^{2\gamma}} \right) &= \frac{d}{dr} \log \left( N(\mf{u},x_0,r) + 1 \right) \\
& +   \frac{d}{dr} \log \left( \frac{H(\mf{u},x_0,r)}{r^{2\gamma}} \right)  \ge 0. \qedhere
\end{align*}
\end{proof}

Finally, we recall a version of the Alt-Caffarelli-Friedman monotonicity formula suited to deal with solutions of \eqref{system simplified}, see Theorem 3.14 in \cite{SoZi} and also Theorem 4.3 in \cite{Wa}. To this aim, we introduce the functionals
\begin{align*}
J_1(r) & := \int_{B_r} \frac{|\nabla u_1|^2 + \beta a_{12} u_1^2 u_2^2}{|x|^{N-2}} \\
J_2(r) & := \int_{B_r} \frac{|\nabla u_2|^2 +  \beta a_{12} u_1^2 u_2^2}{|x|^{N-2}},
\end{align*}
and we define $J(r):= J_1(r) J_2(r) / r^4$.

\begin{proposition}\label{prop: ACF}
Let $\mf{u}$ be a solution of \eqref{system simplified}, with $\Omega \Supset B_R(0)$ for some $R>1$, and let us assume that there exist $\lambda,\mu>0$ such that
\[
\frac{1}{\lambda} \le \frac{ \int_{\pa B_r} u_1^2 }{\int_{\pa B_r} u_2^2} \le \lambda \quad \text{and} \quad \frac{1}{r^{N-1}} \int_{\pa B_r} u_1^2 \ge \mu
\]
for every $r \in [1,R]$. Then there exists $C>0$ depending only on $\lambda,\mu$, and on the dimension $N$, such that
\[
r \mapsto J(r) \exp\{-C (\beta r^2)^{-1/4}\} \quad \text{is non-decreasing for $r \in [1,R]$}.
\] 
\end{proposition}

\subsection{Almgren monotonicity formulae for segregated configurations}\label{sub: segregated configurations}

In \cite[Definition 1.2]{tt}, the authors introduced the sets $\mathcal{G}(\Omega)$ and $\mathcal{G}_{\loc}(\Omega)$, classes of segregated vector valued functions sharing several properties with solutions of competitive systems, including a version of the Almgren monotonicity formula. What is important for us is that, as already recalled in the introduction, if $\{\mf{u}_\beta\}$ is a sequence of solutions of \eqref{main system k comp} (or of the simplified system \eqref{system simplified}) and $\mf{u}_\beta \to \mf{u}$ locally uniformly and in $H^1_{\loc}$, then $\mf{u} \in \mathcal{G}(\Omega)$.

For elements of $\mathcal{G}(\Omega)$, with a slight abuse of notation, let 
\begin{equation}\label{def N segregated}
\begin{split}
  \bullet \quad & E(\mf{v},x_0,r):= \frac{1}{r^{N-2}} \int_{B_r(x_0)} \sum_{i=1}^k |\nabla v_i|^2\\
  \bullet \quad & N(\mf{v},x_0,r):= \frac{E(\mf{v},x_0,r)}{H(\mf{v},x_0,r)} \qquad (\text{Almgren frequency function}),
\end{split}
\end{equation}
where $H$ is defined as in \eqref{def N regular}.


We recall some known facts. The following are a monotonicity formula for functions of $\mathcal{G}(\Omega)$, and a lower estimate of $N(\mf{v},x_0,0^+)$ for points $x_0$ on the free boundary $\{\mf{v}=\mf{0}\}$, for which we refer to \cite[Theorem 2.2 and Corollary 2.7]{tt} and \cite[Lemma 4.2]{SoTe}.

\begin{proposition}\label{prop: monot segregated}
Let $\mf{v} \in \mathcal{G}(\Omega)$. For every $x_0 \in \Omega$ and $r >0$ such that $B_r(x_0) \Subset \Omega$, we have $H(\mf{v},x_0,r) \neq 0$, and the function $N(\mf{v},x_0,r)$ is absolutely continuous and non-decreasing in $r$. Moreover
\[
\frac{d }{d r} \log H(\mf{v},x_0,r) = \frac{2N(\mf{v},x_0,r)}{r},
\]
and $N(\mf{v},x_0,r) \equiv \alpha$ is constant for $r \in (r_1,r_2)$ if and only if $\mf{v}=r^\alpha \mf{g}(\theta)$ is homogeneous of degree $\alpha$ in $\{r_1<|x|<r_2\}$ (here $(r,\theta)$ is a system of polar coordinates centred in $x_0$). Finally, if $x_0 \in \{\mf{v}=\mf{0}\}$, then either $N(\mf{v},x_0,0^+) =1$, or $N(\mf{v},x_0,0^+) \ge 3/2$.
\end{proposition}

\begin{remark}
In \cite[Lemma 4.2]{SoTe} it is shown that the alternative $N(\mf{v},x_0,0^+) =1$, or $N(\mf{v},x_0,0^+) \ge 3/2$, holds for the subclass of $\mathcal{G}_{\loc}(\R^N)$ containing all the homogeneous functions. This is sufficient to have the result in $\mathcal{G}(\Omega)$ for any $\Omega$, and to prove this we argue in the following way: let $\mf{v} \in \mathcal{G}(\Omega)$, not necessarily homogeneous, and let $x_0 \in \{\mf{v}=\mf{0}\}$. Let us consider a normalized blow-up
\[
w_{i,\rho}(x):= \frac{v_i(x_0+\rho x)}{H(\mf{v},x_0,\rho)^{1/2}}.
\]
Up to a subsequence, the family $\{\mf{w}_\rho\}$ is convergent in $\mathcal{C}^{0,\alpha}_{\loc}(\R^N)$ and $H^1_{\loc}(\R^N)$, for $\rho \to 0^+$, to a limit \emph{homoegenous} function $\mf{w} \in \mathcal{G}_{\loc}(\R^N)$ (see Section 3 in \cite{tt}). Thus, for every $r>0$
\begin{align*}
N(\mf{w},0,0^+) &= ( \text{by homogeneity} ) = N(\mf{w},0,r) =  \lim_{\rho \to 0^+} N(\mf{w}_\rho,0, r)\\
&= \lim_{\rho \to 0^+} N(\mf{v},x_0, \rho r) = N(\mf{v},x_0,0^+) ,
\end{align*}
and Lemma 2.7 in \cite{SoTe} applies.
\end{remark}

\begin{remark}\label{rem: homogeneity non-segregated}
It is worth to observe that the characterization ``$N(\mf{v},x_0,r) \equiv \alpha$ is constant for $r \in (r_1,r_2)$ if and only if $\mf{v}=r^\alpha \mf{g}(\theta)$ is homogeneous of degree $\alpha$ in $\{r_1<|x|<r_2\}$" remains true also if $\mf{v}$ is a solution of \eqref{system simplified}. But, since such problem does not admit homogeneous solutions (but constant ones), this means that for any non-constant solution of \eqref{system simplified} the Almgren frequency function is strictly monotone. 
\end{remark}

\subsection{On entire solutions of system \ref{entire system a_{ij}}}\label{sub: entire}

Theorem \ref{thm: blow-up} establishes a relationship between the behaviour of solutions to \eqref{main system k comp} near the interface and the geometry of the solutions to \eqref{entire system a_{ij}}:
\[
\begin{cases}
\Delta V_i = V_i \sum_{j \neq i} a_{ij} V_j^2 \\
V_i \ge 0 
\end{cases} \quad \text{in $\R^N$},
\] 
with $k \ge 2$, $N \ge 1$, and $a_{ij}=a_{ji}$. As stated in the introduction, this relationship will be exploited many times in the rest of the paper, and to this aim we recall some known results concerning existence and classification of solutions to \eqref{entire system a_{ij}}. 


The first trivial observation is that, by the strong maximum principle, the dichotomy $V_i >0$ or $V_i \equiv 0$ in $\R^N$ holds. 

Let us consider now the $k=2$ components system; in such a situation, without loss of generality we can suppose $a_{12}=a_{21}=1$. The $1$-dimensional problem $N =1$ is classified: up to rigid motions and suitable scaling, there exists a unique $1$-dimensional solution satisfying the symmetry property $V_2(x) = V_1(-x)$, the monotonicity condition $V_1'>0$ and $V_2'<0$ in $\R$, and having at most linear growth, see \cite[Lemma 4.1 and Theorem 1.3]{BeLiWeZh} and \cite[Theorem 1.1]{BeTeWaWe}. 

The linear growth is the minimal admissible growth for non-constant solutions of \eqref{entire system a_{ij}}, in the sense that in any dimension $N \ge 1$, if $(V_1,V_2)$ is a nonnegative solution and satisfies the sublinear growth condition
\[
V_1(x)+V_2(x) \le C(1+|x|^\alpha) \qquad \text{in $\R^N$}
\]
for some $\alpha \in (0,1)$ and $C>0$, then one between $V_1$ and $V_2$ is $0$, and the other has to be constant. This \emph{Liouville-type theorem} has been proved by B. Noris et al. in \cite[Propositions 2.6]{NoTaTeVe}.

In contrast to the $1$-dimensional case, already for $N = 2$ the $2$ components system \eqref{entire system a_{ij}} has infinitely many positive solutions with algebraic growth, see \cite{BeTeWaWe}, and also solutions with exponential growth, see \cite{SoZi1}. These existence results was extended also to systems with $k >2$ arbitrary, but only under assumption \eqref{a_ij=1}. Notice that by Theorem \ref{thm: blow-up} solutions with exponential growth cannot be obtained as blow-up limits of sequences $\{\mf{u}_\beta\}$ satisfying \eqref{boundedness} and \eqref{nontrivial}.
We also observe that the existence of solutions in $\R^N$ with $N \ge 3$ which are not obtained by solutions in $\R^2$ has been recently proved in \cite{SoZi2}.

In parallel to the study of the existence, great efforts have been devoted to the research of reasonable conditions which, if satisfied by a solution of \eqref{entire system a_{ij}}, imply the $1$-dimensional symmetry of such solution; this, as explained in \cite{BeLiWeZh}, is inspired by some analogy in the derivation of \eqref{entire system a_{ij}} and of the Allen-Chan equation, for which symmetry results in the spirit of the celebrated De Giorgi's conjecture have been widely investigated. For systems of $k=2$ components, we refer to \cite{Fa}, dealing with monotone solutions in dimension $N=2$; to \cite{FaSo}, where a Gibbons-type conjecture for \eqref{entire system a_{ij}} is proved for any $N \ge 2$; and to \cite{Wa,Wa2}, where it is showed that in any dimension $N \ge 2$, any solution of \eqref{entire system a_{ij}} having linear growth is $1$-dimensional. Writing that $(V_1,V_2)$ has linear growth, we mean that there exists $C>0$ such that
\[
V_1(x)+ V_2(x) \le C (1+|x|) \qquad \forall x \in \R^N.
\]
It is worth to point out that the linear growth condition can be rephrased requiring that $N(\mf{V},0,+\infty) \le  1$, where $N(\mf{V},0,+\infty) = \lim_{r \to +\infty} N(\mf{V},0,r)$ (which exists by monotonicity of the frequency function). 
Other symmetry results for $k=2$ are \cite[Theorem 1.8]{BeLiWeZh} and \cite[Theorem 1.12]{BeTeWaWe}, which are now particular cases of the Wang's results, and the theorems in \cite{Dip}, where stable or monotone solutions with linear growth of more general systems are considered.

Regarding $1$-dimensional symmetry for systems with several components, we refer to \cite[Theorem 1.3]{SoTe}, where for any $k \ge 2$ the authors generalized the main results in \cite{FaSo} and \cite{Wa,Wa2} under assumption \eqref{a_ij=1}. Another important result, which we shall use in the following, is \cite[Corollary 1.9]{SoTe}, where it is proved that if \eqref{a_ij=1} holds and $\mf{V}$ is a non-constant solution to \eqref{entire system a_{ij}}, then
\begin{itemize}
\item either $N(\mf{V},0,+\infty) =1$, and in such case $\mf{V}$ has linear growth, has exactly $2$ nontrivial components, and is $1$-dimensional,
\item or $N(\mf{V},0,+\infty) \ge 3/2$, and hence $\mf{V}$ has not linear growth. In this latter case, adapting Lemma 4.2 in \cite{BeLiWeZh} to systems with several components, it is not difficult to deduce that $\mf{V}$ cannot be $1$-dimensional.
\end{itemize}

To conclude this session, we remark that when $k >2$ but \eqref{a_ij=1} does not hold, it is still possible to recover the classification results in \cite{Wa,Wa2}. Indeed, independently of $a_{ij}$, by \cite[Corollary 1.12]{SoTe} any non-constant solution to \eqref{entire system a_{ij}} having linear growth has exactly $2$-nontrivial components, and hence the results in \cite{Wa,Wa2} are applicable.

\section{Decay estimates I}\label{sec: decay 1}

This section is devoted to the proof of Theorem \ref{thm: global upper estimate} and Proposition \ref{prop: interfaces are good approximation}. Thus, (F1), (F2) and \eqref{boundedness} are in force. We start recalling an important decay estimate which will be frequently used in this paper. 

\begin{lemma}[Lemma 4.4 in \cite{ctv}]\label{lem: decay}
Let $x_0 \in \R^N$ and $r>0$. Let $v \in H^1(B_{2r}(x_0))$ satisfy
\[
\begin{cases}
-\Delta v \le -K v & \text{in $B_{2r}(x_0)$} \\
v \ge 0 & \text{in $B_{2r}(x_0)$} \\
v \le A & \text{on $\pa B_{2r}(x_0)$},
\end{cases}
\]
where $K$ and $A$ are two positive constants. Then for every $\alpha \in (0,1)$ there exists $C_\alpha>0$, not depending on $A$, $K$, $R$ and $x_0$, such that
\[
\sup_{x \in B_r(x_0)} v(x) \le \alpha A  e^{-C_\alpha K^{1/2} r}.
\]
\end{lemma} 

This result, together with the uniform boundedness in the Lipschitz norm of $\{\mf{u}_\beta\}$ (proved in \cite{SoZi}), is the main ingredient in the proof of Theorem \ref{thm: global upper estimate}.

\begin{proof}[Proof of Theorem \ref{thm: global upper estimate}]
For an arbitrary compact $K \Subset \Omega$, let $K'$ be such that $K \Subset K' \Subset \Omega$. By contradiction, we assume that there exist sequences $\beta_n \to +\infty$ and $x_n \in K$ such that 
\[
\beta_n^{1/2} u_{i,n}(x_n) u_{j,n}(x_n) \to +\infty \qquad \text{as $n \to \infty$ for some $i,j =1,\dots,k$},
\]
where $\mf{u}_n= \mf{u}_{\beta_n}$. By compactness, up to a subsequence $x_n \to x_0 \in K$. Moreover, without loss of generality, we can suppose that $i=1$, $j=2$, and 
\begin{equation}\label{ordering}
u_{1,n}(x_n) , u_{2,n}(x_n) \ge u_{h,n}(x_n) \qquad \forall h \neq 1,2.
\end{equation}  
\textbf{Step 1)} \emph{For every $i$, the sequence $(u_{i,n}(x_n))$ converges to $0$ as $n \to \infty$.} \\
As already observed in the introduction, by (F1), (F2) and \eqref{boundedness} we know that $\mf{u}_n \to \mf{u}$ in $\mathcal{C}^0(K')$ and $H^1(K')$. If for instance $u_{1,n}(x_n) \ge 3\delta > 0$ for every $n$, then $u_1(x_0) \ge 2\delta$, and in turn, by the uniform Lipschitz boundedness \cite{SoZi}, $u_{1,n} \ge \delta$ in a neighbourhood $B_{2\rho}(x_0)$. In particular, this implies by (F2) that for any $j \neq 1$
\begin{equation}\label{elliptic inequality}
-\Delta u_{j,n} = - \beta_n a_{ij} u_{i,n}^2 u_{j,n} + f_{j,\beta_n}(x,u_{j,n})  \le (C- C \beta_n) u_{j,n} \le -C \beta_n u_{j,n} 
\end{equation}
in $B_{2\rho}(x_0)$. Since $u_{j,n}$ is also positive and bounded in $L^\infty(B_{2\rho}(x_0))$, uniformly  in $n$, we deduce by Lemma \ref{lem: decay} that
\[
\sup_{B_{\rho}(x_0)} u_{j,n} \le C e^{-C \beta_n^{1/2} \rho} \qquad \forall j=2,\dots,k, \ \forall n.
\]
It follows that
\[
\beta_n^{1/2} u_{1,n}(x_n) u_{2,n}(x_n) \le C \beta_n^{1/2} e^{-C \beta_n^{1/2} \rho} 
\]
for every $n$ sufficiently large, in contradiction with the unboundedness of the left-hand side. 

\medskip

\noindent \textbf{Step 2)} \emph{Conclusion of the proof.} \\
Let $\eps_n:= \sum_{i=1}^k u_{i,n}(x_n)$. By Step 1, $\eps_n \to 0$ as $n \to \infty$. Let 
\[
\tilde u_{i,n}(x):= \frac{1}{\eps_n} u_{i,n}(x_n+\eps_n x) \qquad i=1,\dots,k,
\]
well defined on scaled domains $K_n':= (K'-x_n)/\eps_n$ exhausting $\R^N$ as $n \to \infty$ (here we used the fact that $K \Subset K'$). Note that the normalization has been chosen in such a way that $\sum_i \tilde u_{i,n}(0)= 1$, and the sequence $\{\tilde{\mf{u}}_n\}$ inherits by $\{\mf{u}_n\}$ the uniform boundedness of the Lipschitz semi-norm. As a consequence, $\{\tilde{\mf{u}}_n\}$ is uniformly bounded on compact sets. Now, 
\[
-\Delta \tilde u_{i,n} =  - \eps_n^4\beta_n \tilde u_{i,n} \tilde u_{j,n}^2 -\eps_n f_{i,\beta_n}(x_n +\eps_n x, \eps_n \tilde u_{i,n}(x))  \qquad \text{in $K_n'$},
\]
and the new competition parameter diverges: indeed by assumption
\[
\eps_n^4 \beta_n = \left(\sum_i u_{i,n}(x_n) \right)^4 \beta_n \ge 6 \beta_n u_{1,n}^2(x_n) u_{2,n}^2(x_n) \to +\infty.
\]
Therefore, by \cite{SoTaTeZi} (see also \cite{NoTaTeVe,Wa}) we infer that up to a subsequence $\tilde{\mf{u}}_n \to \tilde{\mf{u}}$ locally uniformly, in $H^1_{\loc}(\R^N)$, and $\tilde u_i \tilde u_j \equiv 0$ in $\R^N$ for every $j \neq i$. Together with the considered normalization, this implies that for instance $\tilde u_1(0) =1$, while $\tilde u_j(0)=0$ for all the other indices $j$. Recalling again the uniform Lipschitz boundedness of $\{\tilde{\mf{u}}_n\}$ on compact sets, for every $n$ sufficiently large we have $\tilde u_{1,n} \ge 1/2$ in a neighbourhood $B_{2\rho}$. By (F1), we finally conclude
\[
-\Delta \tilde u_{j,n} \le - C \eps_n^4 \beta_n \tilde u_{j,n} + C \eps_n^2 \tilde u_{j,n} \le - C \eps_n^4 \beta_n \tilde u_{j,n} \qquad \forall j \neq 1, \ \forall n
\]
in the ball $B_{2\rho}$. Thanks to Lemma \ref{lem: decay}, this implies that
\[
\beta_n u_{1,n}^2 (x_n) u_{2,n}^2(x_n) = \beta_n \eps_n^4 \tilde u_{1,n}^2(0) \tilde u_{2,n}^2(0) \le C \beta_n \eps_n^4 e^{-C \beta_n^{1/2} \eps_n^{2} \rho} \to 0
\]
as $n \to \infty$, a contradiction.
\end{proof}

\begin{remark}
We wish to observe that the uniform Lipschitz boundedness of the sequence $\{\mf{u}_n\}$ is essential in our proof in order to deduce that $\{\tilde{\mf{u}}_n\}$ is locally $L^\infty$-bounded on compact sets of $\R^N$. Notice that the uniform H\"older boundedness would not be sufficient.
\end{remark}
Now we proceed with the:

\begin{proof}[Proof of Proposition \ref{prop: interfaces are good approximation}] If $x_\beta \in \Gamma_\beta$ and $x_\beta \to x_0$, then clearly $x_0 \in \Gamma$ by local uniform convergence $\mf{u}_\beta \to \mf{u} \in \mathcal{G}(\Omega)$. Let now $x_0 \in \Sigma$ with $\mf{u} \not \equiv \mf{0}$, and let us show that there exists $x_\beta \in \Gamma_\beta$ such that $x_\beta \to x_0$. If this is not the case, then $\dist(x_0, \Gamma_\beta) \ge \delta>0$ independently on $\beta$. But then there exists an index $i$ such that (up to a subsequence) 
\[
x_0 \in \{u_{i,\beta}> u_{j,\beta} \text{ for every $j \neq i$}\} \qquad \forall \beta \quad \Longrightarrow \quad B_{\delta/2}(x_0) \subset \{u_i>0\} \cup \Gamma.
\]
It cannot be $B_{\delta/2} (x_0) \cap \mathcal{R} \neq \emptyset$, otherwise we would have self-segregation around the regular part of the free-boundary, in contradiction with \cite[Section 10]{DaWaZh}. This means that $B_{\delta/2}(x_0) \cap \Gamma = B_{\delta/2}(x_0) \cap \Sigma$, being the Hausdorff dimension of $\Sigma$ at most $N-2$ (see \cite{tt} and the previous section). In turn, we deduce that $-\Delta u_i=f_i(x,u_i)$ and $u_i>0$ in $B_{\delta/2}(x_0) \setminus \Sigma$, which implies $-\Delta u_i=f_i(x,u_i)$ in $B_{\delta/2}(x_0)$ since $\Sigma$ has $0$ capacity, and in turn gives $u_i(x_0)>0$ by the strong maximum principle, a contradiction.
\end{proof}

\section{Blow-up and decay estimates II}\label{sec: decay 2}

In the first part of this section we prove Theorem \ref{thm: blow-up}. This, together with Corollaries \ref{thm: classification limits regular part} and \ref{thm: non-simple blow-up}, will be the base point to obtain Theorems-\ref{cor: improved decay singular sequence}. Theorem \ref{thm: decay higher multiplicity} will be the object of the last part of the section. 
%
%

As announced at the end of the introduction, for the sake of simplicity we deal with a sequence of solutions to \eqref{system simplified},
\[
\begin{cases}
\Delta u_{i,\beta} = \beta u_{i,\beta} \sum_{j \neq i} a_{ij} u_{j,\beta}^2 & \text{in $\Omega$} \\
u_{i,\beta} > 0 & \text{in $\Omega$},
\end{cases}
\]
satisfying \eqref{boundedness} and \eqref{nontrivial}: $\{\mf{u}_\beta\}$ is uniformly bounded in $L^\infty(\Omega)$, and is convergent to a nontrivial limit profile $\mf{u} \in \mathcal{G}(\Omega)$. Let $K \Subset \Omega$ be a compact set; then there exists $\bar r>0$ such that $B_{2\bar r}(x) \Subset \Omega$ for every $x \in K$.
Firstly, we derive a simple consequence of assumption \eqref{nontrivial}.

\begin{lemma}\label{lem: N bounded}
There exists $\bar C>0$ such that
\[
E(\mf{u}_\beta,x_0,r) , N(\mf{u}_\beta,x_0,r) \le \bar C
\]
for every $x_0 \in K$, $r \in (0,\bar r]$, and for every $\beta$.
\end{lemma} 
\begin{proof}
Since $E$ is monotone non-decreasing as function of $r$, for the first part of the thesis it is sufficient to bound $E(\mf{u}_\beta,x_0,\bar r)$ uniformly in $x_0$ and $\beta$. 
This can be done as in point (6) of Lemma 2.1 in \cite{SoZi} (in the present setting it is actually easier), and we only sketch the proof for the sake of completeness. Let $\varphi \in \mathcal{C}^\infty_c(B_{2\bar r})$, with $\varphi \equiv 1$ in $B_{\bar r}$ and $0 \le \varphi \le 1$. Let $\varphi_{x_0}(y):= \varphi(x-x_0)$. By testing the equation in \eqref{system simplified} with $\varphi_{x_0}$, we can show that
\[
\int_{B_{\bar r}(x_0)} \beta u_{i,\beta} \sum_{j \neq i} a_{ij} u_{j,\beta}^2 \le C
\]
for some $C>0$ independent of $x_0$ and $\beta$. This, together with assumption \eqref{boundedness}, gives the boundedness of 
\[
\int_{B_{\bar r}(x_0)}\sum_{j \neq i} a_{ij} u_{i,\beta}^2 u_{j,\beta}^2.
\]
To control the integrals of the square of the gradient, we test the equation in \eqref{system simplified} with $u_{i,\beta} \varphi_{x_0}^2$, and obtain the desired estimate after some integration by parts.

Concerning the boundedness of the Almgren quotient, by monotonicity again (Proposition \ref{prop: almgren}) it is sufficient to check that $N(\mf{u}_\beta,x_0,\bar r)$ is bounded uniformly in $x_0$ and $\beta$. Thanks to the first part, it is equivalent to prove that $H(\mf{u}_\beta,x_0,\bar r) \ge C >0$ independently of $x_0$ and $\beta$. If this is not true, then there exist sequences $\beta \to +\infty$ and $x_\beta \in K$ such that $H(\mf{u}_\beta,x_\beta,\bar r) \to 0$. On the other hand, observing that $x_\beta \to x_0 \in K$ since $K$ is compact, by uniform convergence we have $H(\mf{u}_\beta,x_\beta,\bar r) \to H(\mf{u},x_0,\bar r)$. This is a strictly positive quantity, as ensured by Proposition \ref{prop: monot segregated} and assumption \eqref{nontrivial}, and hence we reached the desired contradiction. 
%
\end{proof}

The following statement suggests the proper choice of $r_\beta$ in Theorem \ref{thm: blow-up}.

\begin{lemma}\label{lem: choice of r}
For any $x_0 \in K$ and $\beta > 0$ sufficiently large, there exists a unique $r_\beta(x_0)> 0$ such that
\[
	\beta H(\mf{u}_\beta, x_0, r_\beta(x_0)) r_\beta(x_0)^2 = 1.
\]
Moreover, let $\{x_\beta\} \subset K$. Then $r_\beta(x_\beta) \to 0$, and consequently
\[
	\frac{\Omega - x_\beta}{r_\beta(x_\beta)} \to \R^N \qquad \text{as $\beta \to +\infty$},
\]
in the sense that for any $R>0$ there exists $\bar \beta$ sufficiently large such that $B_R \subset (\Omega - x_\beta)/r_\beta(x_\beta)$ provided $\beta > \bar \beta$.
\end{lemma}
\begin{proof}
First of all, by Proposition \ref{prop: almgren} 
\[
	r \mapsto \beta H(\mf{u}_\beta,x_0,r) r^2
\]
is increasing for any $x_0 \in K$ and $\beta$ fixed. Since $H(\mf{u}_\beta,x_0,\bar r) \to H(\mf{u},x_0,\bar r)$ and assumption \eqref{nontrivial} is in force, we have
\[
\beta H(\mf{u}_\beta,x_0,\bar r) \bar r^2 >1 \qquad \forall \beta>\bar \beta.
\] 
Moreover, since $\mf{u}_\beta$ is a vector valued smooth function with positive components, it results
\[
	\lim_{r \to 0^+}  \beta H(\mf{u}_\beta,x_0,r) r^2 = 0,
\]
and hence the thesis follows by the mean value theorem. 
For the second part of the lemma we argue by contradiction, assuming that for a sequence $x_\beta \in K$ it results $r_\beta(x_\beta) \ge \tilde r>0$. By compactness $x_\beta \to x_0 \in K$, and thanks to \eqref{nontrivial} and uniform convergence 
\[
1 = \beta H(\mf{u}_\beta,x_\beta,r_\beta(x_\beta)) r_\beta(x_\beta)^2 \ge \frac{\beta}{2} H(\mf{u},x_0,\tilde r) \tilde r^2 \to +\infty
\]
as $\beta \to +\infty$, a contradiction.
\end{proof}


With the previous lemmas in hand we can proceed with the proof of Theorem \ref{thm: blow-up}. Before, we observe that by definition
\begin{equation}\label{scaled quantities}
\begin{split}
H(\mf{v}_\beta,0,r) & = \frac{H(\mf{u}_\beta,x_\beta,r_\beta r)}{H(\mf{u}_\beta,x_\beta,r_\beta)} 		\\
E(\mf{v}_\beta,0,r) & = \frac{E(\mf{u}_\beta,x_\beta,r_\beta r)}{H(\mf{u}_\beta,x_\beta,r_\beta)}  \\
N(\mf{v}_\beta,0,r) & = N(\mf{u}_\beta,x_\beta,r_\beta,r).
\end{split}
\end{equation}

\begin{proof}[Proof of Theorem \ref{thm: blow-up}]
Let us consider the scaled sequence $\mf{v}_\beta$:
\[
v_{i,\beta}(x):= \frac{u_{i,\beta}(x_\beta + r_\beta x)}{H(\mf{u}_\beta,x_\beta,r_\beta)^{1/2}}
\]
where $r_\beta=r_\beta(x_\beta)$ is given by Lemma \ref{lem: choice of r}, and we recall that $x_\beta$ is a sequence of points on the interfaces $\Gamma_\beta$.
Thanks to the choice of $r_\beta$
\begin{equation}\label{scaled equation}
\Delta v_{i,\beta}(x) = v_{i,\beta} \sum_{j \neq i} a_{ij} v_{j,\beta}^2  \qquad \text{in } \frac{\Omega-x_\beta}{r_\beta},
\end{equation}
and moreover by \eqref{scaled quantities} and Lemma \ref{lem: N bounded} 
\[
N(\mf{v}_\beta,0,r) \le \bar C \qquad \forall r \le \frac{\bar r}{r_\beta},
\]
where we recall that $\bar r>0$ has been chosen so that $B_{2\bar r}(x) \Subset \Omega$ for every $x \in K$.
The previous estimate, together with Proposition \ref{prop: almgren}, implies that 
\[
\frac{d}{dr}\log H(\mf{v}_\beta,0,r) = \frac{2 N(\mf{v}_\beta,0,r)}{r} \le \frac{2\bar C}{r} \qquad \forall r \le \frac{\bar r}{r_\beta};
\]
by integrating
\[
H(\mf{v}_\beta,0,r) \le H(\mf{v}_\beta,0,1) r^{2 \bar C} \qquad \forall r \in \left[1,\frac{\bar r}{r_\beta}\right].
\]
Consequentlym for any fixed $r>1$, the sequence $\{H(\mf{v}_\beta,0,r)\}_\beta$ is bounded, and since $\{N(\mf{v}_\beta,0,r)\}_\beta$ is also bounded by Lemma \ref{lem: N bounded}, we infer that $\{E(\mf{v}_\beta,0,r)\}_\beta$ is in turn bounded. Using a Poincar\'e inequality, it is not difficult to deduce that this gives boundedness of $\{\mf{v}_\beta\}$ in $H^1(B_r)$, and hence also in $L^2(\partial B_r)$. By subharmonicity, $\{\mf{v}_\beta\}$ is then $L^\infty$-bounded in any compact set of $B_r$, and, by regularity theory for elliptic equations, this provides $\mathcal{C}^2_{\loc}(B_r)$ convergence to a limit $\mf{V}^{(r)}$, solution to \eqref{entire system a_{ij}} in $B_r$. Since in the previous argument $r>1$ has been arbitrarily chosen, we can take a sequence of radii diverging to $+\infty$, and with a diagonal selection we obtain $\mathcal{C}^2_{\loc}$ convergence to an entire limit profile $\mf{V}$. Notice that $\mf{V}$ has at least two nontrivial components. Indeed, by definition of $\Gamma_\beta$, we known that $u_{i_1,\beta}(x_\beta)= u_{i_2,\beta}(x_\beta) \ge u_{j,\beta}(x_\beta)$ for at least two indices $i_1 \neq i_2$, for all $j$. This implies that $v_{i_1,\beta}(0) = v_{i_2,\beta}(0) \ge v_{j,\beta}(0)$. Now, it is easy to check that $v_{i_1}(0) \ge C >0$: if not, then $V_j(0) = 0$ for all $j$, and since $V_j$ is nonnegative and solves
\[
	\Delta V_j = V_j \sum_{i \neq j} a_{ij} V_i^2 \qquad \text{in $\R^N$}
\]
by the strong maximum principle $V_{j} \equiv 0$ in $\R^N$ for all $j$, in contradiction with the fact that 
\[
	\int_{\pa B_1} \sum_{i=1}^k v_{i,\beta}^2 = 1. \qedhere
\]
\end{proof}

\begin{remark}\label{rem: conv also not on gamma}
If $x_\beta$ is not necessarily a sequence in $\Gamma_\beta$, the previous proof establishes that the scaled sequence $\mf{v}_\beta$ is convergent to a limit $\mf{V} \not \equiv \mf{0}$. It is worth to point out that in case $x_\beta \not \in \Gamma_\beta$ for $\beta$ large, such convergence is not really informative, since the limit function will have only $1$ nontrivial components, being a constant, and all the others will be $0$.
\end{remark}

\begin{remark}\label{rem: con f 1}
It is worth to observe explicitly that the limit system does not change in presence of nontrivial $f_{i,\beta}(x,u_{i,\beta})$. Indeed, the transformed nonlinearities appearing in \eqref{scaled equation} takes the form
\[
\frac{r_\beta^2}{H(\mf{u}_\beta,x_\beta,r_\beta)}   f_{i,\beta}\left( x_\beta + r_\beta x, u_{i,\beta}(x_\beta+r_\beta x) \right),
\]
and by (F1) can be easily controlled by 
\[
\frac{r_\beta^2 u_{i,\beta}(x_\beta+r_\beta x)   }{H(\mf{u}_\beta,x_\beta,r_\beta)} = r_\beta^2 v_{i,\beta}(x).
\]
Therefore, once the local $L^\infty$ boundedness of $\{\mf{v}_\beta\}$ is proved (instead of the subharmonicity, one can use a Brezis-Kato argument), the transformed nonlinearities converge to $0$ locally uniformly since $r_\beta \to 0$.

In the same spirit, we observe that if $f_{i,\beta} \not \equiv 0$ and (F1) holds, then both $H(\mf{u}_\beta,x_\beta,\cdot)$ and $N(\mf{u}_\beta,x_\beta,\cdot)$ are not necessarily monotone, but satisfy some almost monotonicity formulae, see Proposition 3.5 in \cite{SoZi}. Using such a result and refining a little bit the previous computations, it is not difficult to check that what we proved in this subsection hold also in that context, as stated in the introduction.
\end{remark}

\subsection{Lower estimates on the decay}

We aim at proving Theorem \ref{thm: lower general}. As a first step, we relate the value of $\mf{u}_\beta$ on the interface with $H(\mf{u}_\beta,x_\beta,r_\beta)$. 

\begin{lemma}\label{lem: H con m}
Let $\{x_\beta\} \subset K$, and let $r_\beta=r_\beta(x_\beta)$ be defined by Lemma \ref{lem: choice of r}. There exists $C>1$ such that
\[
\frac{1}{C} \left(\sum_{i=1}^k u_{i,\beta}(x) \right)^2 \le H(\mf{u}_\beta,x_\beta,r_\beta) \le C \left(\sum_{i=1}^k u_{i,\beta}(x) \right)^2.
\] 
\end{lemma}
\begin{proof}
By Theorem \ref{thm: blow-up} (see also Remark \ref{rem: conv also not on gamma}) we know that $\mf{v}_\beta \to \mf{V} \not \equiv \mf{0}$. Thus there exists $C>0$ such that
\[
	\frac{1}{C} \le   \sum_{i=1}^{k} v_{i,\beta}^2(0) \le C.
\]
Recalling the definition of $\mf{v}_\beta$ and using the triangular inequality, we obtain the desired result.
\end{proof}


We are ready to proceed with the: 
\begin{proof}[Proof of Theorem \ref{thm: lower general}]
Let $x_\beta \in \Gamma_\beta$, and let $r_\beta=r_\beta(x_\beta)$ be defined by Lemma \ref{lem: choice of r}. Let $\eps>0$ be fixed. By the continuity of the Almgren frequency function there exists $\bar r>0$ such that $N(\mf{u},x_0,\bar r) \le D+2\eps$. By convergence $N(\mf{u}_\beta,x_\beta,\bar r) \le D+\eps$ at least for $\beta$ sufficiently large, and as a consequence of Proposition \ref{prop: almgren} we deduce that
\[
N(\mf{u}_\beta,x_\beta, r) \le D+\eps \qquad \forall r \leq \bar r.
\]
Therefore
\[
\frac{d}{dr} \log H(\mf{u}_\beta,x_\beta,r) = \frac{2N(\mf{u}_\beta,x_\beta,r)}{r} \le \frac{2(D+\eps)}{r} \qquad \forall r \in (0,\bar r],
\]
which by integration implies that
\[
C_\eps \le \frac{H(\mf{u}_\beta,x_\beta,\bar r)}{\bar r^{2(D+\eps)}} \le \frac{H(\mf{u}_\beta,x_\beta, r)}{r^{2(D+\eps)}} \qquad \forall r \in (0,\bar r],
\]
with $C_\eps>0$ by \eqref{nontrivial}. In particular, recalling that $r_\beta \to 0$, this estimate holds for $r=r_\beta$, at least for $\beta$ sufficiently large. But then, thanks to Lemma \ref{lem: H con m} and the choice of $r_\beta$, we obtain
\begin{align*}
\left(\sum_{i=1}^k u_{i,\beta}(x_\beta) \right)^2  & \ge C_\eps r_\beta^{2(D+\eps)} \ge \frac{C_\eps}{\beta^{(D+\eps)}  H(\mf{u}_\beta,x_\beta,r_\beta)^{(D+\eps)}} \\
& \ge \frac{C_\eps }{\beta^{(D+\eps)} \left(\sum_{i=1}^k u_{i,\beta}(x_\beta) \right)^{2(D+\eps)}  }, 
\end{align*}
that is
\[
\beta^{D+\eps} \left(\sum_{i=1}^k u_{i,\beta}(x_\beta) \right)^{2(1+D+\eps)} \ge C_\eps,
\]
whence the thesis follows.
\end{proof}

As a corollary:
\begin{proof}[Proof of Theorem \ref{thm: lowe estimate}]
If $x_0 \in \mathcal{R}$, then Theorem \ref{thm: lower general} holds with $D=1$, or equivalently
\[
\liminf_{\beta \to +\infty} \beta^{1/4+\eps} \sum_{i=1}^k u_{i,\beta}(x_\beta) \ge C_\eps
\]
The thesis is then a consequence of this estimate and Theorem \ref{corol: decay components vanishing}, which will be proved in the next section with an independent argument.
\end{proof}

\begin{remark}
If $f_{i,\beta} \not \equiv 0$, then we know that $N(\mf{u}_\beta,x_\beta,r)$ is not necessarily monotone in $r$. But, using Proposition 3.5 in \cite{SoZi}, we have however that there exists $C>0$ independent of $\beta$ (for this we use (F1) and \eqref{boundedness}) such that $(N(\mf{u}_\beta,x_\beta,r) +1)\exp\{C r\}$ is non-decreasing in $r$. This allows to prove Theorem \ref{thm: lower general} in the following way: for $\eps>0$, we firstly fix $\rho>0$ such that
\[
(N(\mf{u},x_0,\rho)+1)e^{C \rho} \le (D+2\eps+1).
\]
Since $N(\mf{u},x_0,0^+)=D+\eps$, this is possible. At least for $\beta$ sufficiently large, this choice implies that $N(\mf{u}_\beta,x_\beta,r) \le D+\eps$ for any $0<r<\rho$ and $\beta$. As a consequence, we can proceed with the proof of Theorem \ref{thm: lower general} without further changes.
\end{remark}

\subsection{Further consequences of the existence of non-segregated blow-up}

In the rest of the section we keep the notation introduced in the proof of Theorem \ref{thm: blow-up}: $\{\mf{u}_\beta\}$ denotes the original sequence with limit $\mf{u} \not \equiv \mf{0}$ in $\mathcal{G}_{\loc}(\Omega)$, $\{\mf{v}_\beta\}$ denotes the scaled sequence defined in the quoted statement, which is converging in $\mathcal{C}^2_{\loc}(\R^N)$ to a limit $\mf{V}$, solution to \eqref{entire system a_{ij}} with at least two non-trivial components. As reviewed in Section \ref{sec: prel}, a relevant quantity to understand the properties of $\mf{V}$ is the limit at infinity of the Almgren frequency. Let
\[
d:= \lim_{r \to +\infty} N(\mf{V},0,r), \quad \text{and} \quad D:= \lim_{r \to 0^+} N(\mf{u},x_0,0^+).
\]

\begin{lemma}
In the previous notation, $d \le D$.
\end{lemma}
\begin{proof}
By the convergence of $\{\mf{u}_\beta\}$ and $\{\mf{v}_\beta\}$, together with the monotonicity of the Almgren frequency function (see Proposition \ref{prop: almgren}), we have that for any $r,\rho>0$
\begin{align*}
N(\mf{V},0,r) &= \lim_{\beta \to +\infty} N(\mf{v}_\beta,0,r) = \lim_{\beta \to +\infty} N(\mf{u}_\beta,x_\beta,r_\beta r) \\
&\le \lim_{\beta \to +\infty} N(\mf{u}_\beta, x_\beta, \rho) = N(\mf{u},x_0,\rho)
\end{align*}
(notice that, for any $r,\rho>0$ it results that $r_\beta r \le \rho$ for $\beta$ sufficiently large). Passing to the limit as $r \to +\infty$ and $\rho \to 0^+$, we obtain the desired result. 
\end{proof}

Let us point out that the previous result is somehow sharp: without further technical assumptions, it is not possible to show that $d = D$ for a generic point $x_0$. Actually, it is possible to construct counterexamples with $d < D$: this can be done considering suitable translations of the original functions $\{\mf{u}_\beta\}$, so that the macroscopic scale and the blow-up scale behave in a different way. 

In any case, the previous lemma has two direct consequences:

\begin{proof}[Proof of Corollary \ref{thm: classification limits regular part}]
If $x_0 \in \mathcal{R}$, then by definition $D=1$. Therefore, the thesis is a simple consequence of the uniqueness of solutions of \eqref{entire system a_{ij}} with $N(\mf{V},0,+\infty) \le 1$ and having at least two non-trivial components (see the main results in \cite{Wa,Wa2} for $k=2$, and Theorem 1.3 in \cite{SoTe} for an arbitrary $k \ge 2$). 
\end{proof}

\begin{proof}[Proof of Corollary \ref{thm: non-simple blow-up}]
If $x_\beta \in \Sigma_\beta$, then we can show that $0$ is a singular point for the function $\mf{V}$, in the following sense: denoting the interface of $\mf{V}$ by $\Gamma_{\mf{V}}$, and its singular part by 
$\Sigma_{\mf{V}}$ (see Definitions \ref{def interface} and \ref{def singular interface}), we prove that since $x_\beta \in \Sigma_\beta$, then $\mf{0} \in \Sigma_{\mf{V}}$. Notice that by definition of $\Sigma_\beta$ there are two possibilities: either (up to a subsequence) there exist at least $3$ distinct indices such that
\begin{equation}\label{singular condition 1}
u_{i_1,\beta}(x_\beta) = u_{i_2,\beta}(x_\beta) = u_{i_3,\beta}(x_\beta) \ge u_{j,\beta}(x_\beta) \qquad \forall \beta, \ \forall j,
\end{equation}
or (up to a subsequence) there exist two indices $i_1$ and $i_2$ such that
\begin{equation}\label{singular condition 2}
\begin{cases}
u_{i_1,\beta}(x_\beta) = u_{i_2,\beta}(x_\beta) >  u_{j,\beta}(x_\beta) \quad \forall j \neq i_1,i_2,\\
\nabla (u_{i_1,\beta}-u_{i_2,\beta})(x_\beta) = 0 
\end{cases} \qquad \forall \beta.
\end{equation}
If \eqref{singular condition 1} is in force, then 
\[
v_{i_1,\beta}(0) = v_{i_2,\beta}(0) = v_{i_3,\beta}(0) \ge v_{j,\beta}(0) \qquad \forall \beta, \forall j,
\]
while if \eqref{singular condition 2} holds, then
\[
\begin{cases}
v_{i_1,\beta}(0) = v_{i_2,\beta}(0) > v_{j,\beta}(0) \quad \forall j \neq i_1,i_2,\\
\nabla (v_{i_1,\beta}-v_{i_2,\beta})(0) = 0 
\end{cases} \qquad \forall \beta.
\]
Recalling that $\mf{v}_\beta \to \mf{V}$ in $\mathcal{C}^2_{\loc}(\R^N)$, we infer that in any case $0 \in \Sigma_{\mf{V}}$. Now, if $\mf{V}$ is $1$-dimensional, then $\Gamma_{\mf{V}}$ is a hyperplane and $\Sigma_{\mf{V}} = \emptyset$. Therefore, $\mf{V}$ cannot be $1$-dimensional. Since $\mf{V}$ is not $1$-dimensional, by \cite[Theorem 1.3 and Corollary 1.9]{SoTe} we have $3/2 \le d$, and since $d \le D$ we conclude that $x_0 \in \Sigma$.
\end{proof}

We now proceed with proof of Theorem \ref{cor: improved decay singular sequence}. 

\begin{proposition}\label{prop 151}
Under assumption  \eqref{nontrivial}, let $\mf{V}$ be the limit profile given by Theorem \ref{thm: blow-up}, and let $d = N(\mf{V},0,+\infty)$. For any $\eps>0$ there exists $C_\eps>0$ such that 
\[
\limsup_{\beta \to +\infty} \beta^{(d-\eps)/(2+2d)}  \left(\sum_{i=1}^k u_{i,\beta}(x_\beta) \right) \le C_\eps.
\]
\end{proposition}
\begin{proof}
We study the monotonicity of the function
\[
r \mapsto \frac{H(\mf{u}_\beta,x_\beta,r)}{r^{2N(\mf{u}_\beta,x_\beta,r)}} \qquad r \in (0,\bar r],
\]
where we recall that $\bar r>0$  has been chosen so that $B_{2\bar r}(x) \Subset \Omega$ for every $x \in K$. Recalling Proposition \ref{prop: almgren}, we have
\begin{align*}
\frac{d}{dr} \log \frac{H(\mf{u}_\beta,x_\beta,r)}{r^{2N(\mf{u}_\beta,x_\beta,r)}} &= \frac{d}{dr} \log H(\mf{u}_\beta,x_\beta,r) - \frac{d}{dr}\left( 2N(\mf{u}_\beta,x_\beta,r) \log r \right)  \\
&  = \frac{2N(\mf{u}_\beta,x_\beta,r)}{r} - \frac{2N(\mf{u}_\beta,x_\beta,r)}{r} - 2(\log r) \frac{d}{dr} N(\mf{u}_\beta,x_\beta,r)  \\
& \ge 0 \qquad \forall r \in (0,\bar r].
\end{align*}
Therefore, using also the boundedness of $\{\mf{u}_\beta\}$, we infer that
\[
H(\mf{u}_\beta,x_\beta,r) \le H(\mf{u}_\beta,x_\beta,\bar r) r^{2N(\mf{u}_\beta,x_\beta,r)} \le C r^{2N(\mf{u}_\beta,x_\beta,r)} \qquad \forall r \in (0,\bar r], 
\]
Now, since $d= N(\mf{V},0,+\infty)$, for any $\eps>0$ there exists $\rho=\rho(\eps)>1$ sufficiently large such that $N(\mf{V},0,\rho)=d-\eps/2$, and hence by convergence $d-\eps \le N(\mf{v}_\beta,0,\rho) \le d$ for $\beta$ sufficiently large. With this choice of $\rho$, we observe that always for $\beta$ large
\begin{equation}\label{stima 151}
H(\mf{u}_\beta,x_\beta,\rho r_\beta) \le  C (\rho r_\beta)^{2N(\mf{u}_\beta,x_\beta, \rho r_\beta)},
\end{equation}
as $r_\beta \to 0$ as $\beta \to +\infty$. 
The left hand side can be controlled from below as follows:
\begin{equation}\label{stima basso 151}
\left(\sum_{i=1}^k u_{i,\beta}(x_\beta) \right)^2  \le C H(\mf{u}_\beta,x_\beta,r_\beta) \le C H(\mf{u}_\beta,x_\beta,r_\beta \rho),
\end{equation}
where we used Lemma \ref{lem: H con m} and the monotonicity of $H$ (recall that $\rho>1$). To control the right hand side in \eqref{stima 151}, we recall \eqref{scaled quantities} and that $N(\mf{v}_\beta,0,\rho) \le d$ for every $\beta$ large, so that
\[
 C (\rho r_\beta)^{2N(\mf{u}_\beta,x_\beta, \rho r_\beta)} \le C \rho^{2d} r_\beta^{2N(\mf{v}_\beta,0, \rho)} \le C_\eps r_\beta^{2N(\mf{v}_\beta,0, \rho)},
 \]
 where the dependence of $C$ on $\eps$ comes from the dependence $\rho=\rho(\eps)$. By definition of $r_\beta$ (see Lemma \ref{lem: choice of r}), the previous estimate implies that
\begin{equation}\label{stima alto 151}
\begin{split}
 C (\rho r_\beta)^{2N(\mf{u}_\beta,x_\beta, \rho r_\beta)} &\le \frac{C_\eps}{ \beta^{N(\mf{v}_\beta,0, \rho)} H(\mf{u}_\beta,x_\beta,r_\beta)^{N(\mf{v}_\beta,0, \rho) }  } \\
 & \le \frac{C_\eps}{ \beta^{N(\mf{v}_\beta,0, \rho)} \left( \sum_{i=1}^k u_{i,\beta}(x_\beta) \right)^{2N(\mf{v}_\beta,0, \rho)}  },
 \end{split}
\end{equation}
where in the last step we used Lemma \ref{lem: H con m}. Collecting \eqref{stima basso 151} and \eqref{stima alto 151}, and coming back to \eqref{stima 151}, we conclude that 
\[
\beta^{d-\eps}  \left(\sum_{i=1}^k u_{i,\beta}(x_\beta) \right)^{2+2d} \le \beta^{N(\mf{v}_\beta,0, \rho)}  \left(\sum_{i=1}^k u_{i,\beta}(x_\beta) \right)^{2+ 2 N(\mf{v}_\beta,0, \rho)} \le C_\eps
\]
for some $C_\eps>0$ independent of $\beta$. 
\end{proof}

As a consequence:

\begin{proof}[Proof of Theorem \ref{cor: improved decay singular sequence}]
Let $\mf{V}$ be the limit profile defined in Theorem \ref{thm: blow-up}. By Corollary \ref{thm: non-simple blow-up} it is not $1$-dimensional, and hence by \cite[Theorem 1.3 and Corollary 1.9]{SoTe} it results $3/2 \le d$. Since the function $d \mapsto (d-\eps)/(2+2d)$ is strictly increasing for $d \ge 1$, this together with Proposition \ref{prop 151} gives the desired result.
\end{proof}

\begin{remark}
As in the previous subsections, we point out that replacing Proposition \ref{prop: almgren} with Proposition 3.5 in \cite{SoZi} and refining the computations, it is not difficult to extend the above proofs in case $f_{i,\beta} \not \equiv 0$.
\end{remark}

\subsection{General decay estimate around singular points}

In this subsection we prove Theorem \ref{thm: decay higher multiplicity}. Let us fix $x_0 \in \Sigma$, so that by definition $D:= N(\mf{u},x_0,0^+)>1$, and let $0<\eps<D-1$ be arbitrarily chosen. Using the notation introduced in Theorem \ref{thm: blow-up}, let $d:= N(\mf{V},0,+\infty)$. If $d>1$, then we can proceed as in Corollary \ref{cor: improved decay singular sequence}, whose thesis is in fact stronger than the one considered here. Thus, we have only to examine the case $d=1$: we recall that this means that $\mf{V}$ is the only $1$-dimensional solution of \eqref{entire system a_{ij}}, \emph{having linear growth}. Let us introduce
\[
\begin{split}
R_\beta &:= \inf\left\{r>0: N(\mf{u}_\beta,x_\beta,r) >D-\eps \right\} \\
\rho_\beta& := \inf\left\{r>0: N(\mf{u}_\beta,x_\beta,r) >1 \right\}. 
\end{split}
\]
Let $\bar r>0$ be such that $B_{2 \bar r}(x) \Subset \Omega$ for any $x \in K$. Recall now that $N(\mf{u}_\beta,x_\beta,\cdot)$ is non-decreasing. Thus, observing that for any $r \in (0,\bar r]$ one has $N(\mf{u}_\beta,x_\beta,r) \to N(\mf{u},x_0,r) \ge D$ as $\beta \to +\infty$, while $N(\mf{u}_\beta,x_\beta,0^+)= 0$ for any $\beta$ fixed, we deduce that $\rho_\beta$ and $R_\beta$ are positive real numbers, and $0<\rho_\beta<R_\beta \to 0^+$. 

With the notation of Theorem \ref{thm: blow-up}, let 
\[
\begin{split}
\bar R_\beta &:= \inf\left\{r>0: N(\mf{v}_\beta,0,r) >D-\eps \right\} \\
\bar \rho_\beta& := \inf\left\{r>0: N(\mf{v}_\beta,0,r) >1 \right\}. 
\end{split}
\]
Notice that, by definition and \eqref{scaled quantities}, one has 
\begin{equation}\label{rapporti scales}
\bar R_\beta = \frac{R_\beta}{r_\beta} \quad \text{and} \quad \bar \rho_\beta = \frac{\rho_\beta}{r_\beta}.
\end{equation}
Moreover, recall that $\mf{v}_\beta$ is defined in a domain containing the ball $B_{\bar r/r_\beta}$. 

Having introduced $\bar \rho_\beta$ and $R_\beta$, we can now borrow some ideas from the proof of Theorem 1.3 in \cite{SoZi}, see Section 4 therein. We shall carry some information through the different scales  $1<\rho_\beta<R_\beta<\bar r/r_\beta$. In doing so, we shall use three different monotonicity formulae, one from each interval $(1,\rho_\beta)$, $(\rho_\beta,R_\beta)$, $(R_\beta, \bar r/r_\beta)$, whose validity rests essentially on the corresponding estimate on $N(\mf{v}_\beta,0,\cdot)$.

\begin{lemma}
It results that $\bar \rho_\beta, \bar R_\beta \to +\infty$ as $\beta \to +\infty$.
\end{lemma}
\begin{proof}
Since by definition $\bar \rho_\beta \le \bar R_\beta$, it is sufficient to check that $\bar \rho_\beta \to +\infty$. This is a simple consequence of the convergence $\mf{v}_\beta \to \mf{V}$, and of the fac that $N(\mf{V},0,r)  \le 1$ for every $r>0$. As observed in Remark \ref{rem: homogeneity non-segregated}, since $\mf{V}$ solves \eqref{entire system a_{ij}} this implies that $N(\mf{V},0,r) <1$ for every $r >0$. Therefore, if by contradiction we suppose that $\{\bar \rho_\beta\}$ is bounded, we obtain up to a subsequence $\bar \rho_\beta \to \bar \rho$, and hence
\[
N(\mf{V},0,\bar \rho) = \lim_{\beta \to +\infty} N(\mf{v}_\beta,0,\bar \rho_\beta) =1,
\]
a contradiction.
\end{proof}

By definition and by Proposition \ref{prop: monot e+h}, in the intervals $(\bar \rho_\beta,\bar R_\beta)$ and $(\bar R_\beta, \bar r/r_\beta)$ we have two powerful monotonicity formulae. In $(1,\bar \rho_\beta)$ we do not have any estimate from below on the Almgren's frequency, and hence we shall use a perturbed Alt-Caffarelli-Friedman monotonicity formula. To this end, we recall again that $\mf{v}_\beta \to \mf{V}$ in $\mathcal{C}^2_{\loc}(\R^N)$, and $\mf{V}$ the unique $1$-dimensional solution of \eqref{entire system a_{ij}}, which have exactly two non-trivial components. Up to a relabelling, it is not restrictive to suppose that $V_1,V_2 \not \equiv 0$, so that we are naturally led to consider $J_\beta(r) :=r^{-4} J_{1,\beta}(r) \cdot J_{2,\beta}(r)$, where
\begin{align*}
J_{1,\beta}(r) &:= \int_{B_r} \left(\left|\nabla v_{1,\beta}\right|^2 + a_{12} v_{1,\beta}^2 v_{2,\beta}^2 \right)|x|^{2-N}  \\
J_{2,\beta}(r)& := \int_{B_r} \left(\left|\nabla v_{2,\beta}\right|^2 + a_{12} v_{1,\beta}^2 v_{2,\beta}^2\right)|x|^{2-N}.
\end{align*}
The validity of the following monotonicity formula will be the key in our concluding argument.

\begin{lemma}\label{lem: ACF uniform}
There exists $C>0$ independent of $\beta$ such that $J_{1,\beta}(r) \ge C$ and $J_{2,\beta}(r) \ge C$ for every $r \in [1,\bar \rho_\beta/3]$, and
\[
r \mapsto J_{\beta}(r) e^{-C r^{-1/2}} \quad \text{is non-decreasing for $r \in [1,\bar \rho_\beta/3]$}.
\]
\end{lemma}

The proof consists in checking that the assumptions of Proposition \ref{prop: ACF} are satisfied by $(v_{1,\beta},v_{2,\beta})$ uniformly in $\beta$: that is, one has to show that there exist $\lambda,\mu>0$ such that
\[
\frac{1}{\lambda} \le \frac{ \int_{\pa B_r} v_{1,\beta}^2 }{\int_{\pa B_r} v_{2,\beta}^2} \le \lambda \quad \text{and} \quad \frac{1}{r^{N-1}} \int_{\pa B_r} v_{1,\beta}^2 \ge \mu
\]
for every $r \in [1,\bar \rho_\beta/3]$, for every $\beta$. This can be done arguing exactly as in the proof of Lemma 4.9 in \cite{SoZi} (actually the proof is easier in the present setting, since we neglected the nonlinearities $f_{i,\beta}$), see Section 4.1 in the quoted paper, and thus we omit the details. We emphasize that there we only used the fact that $(v_{1,\beta},v_{2,\beta}) \to (V_1,V_2)$ locally uniformly and in $H^1_{\loc}(\R^N)$, with $V_1,V_2 \not \equiv 0$, and the control $N(\mf{v}_\beta,0,r) \le 1$ for $r \le \bar \rho_\beta/3$. Both these properties are satisfied in the present setting.

With Lemma \ref{lem: ACF uniform} in hands, we can proceed with the:

\begin{proof}[Conclusion of the proof of Theorem \ref{thm: decay higher multiplicity}]
By Lemma \ref{lem: ACF uniform}, and since $(V_1,V_2)$ are two positive non-constant functions, for some $C>0$ we have
\begin{equation}\label{chain 1}
C \le J_\beta(1) e^{-C} \le J_\beta\left(\frac{\bar \rho_\beta}{3} \right)e^{-C \bar \rho_\beta^{-1/2}} \le C  J_\beta\left(\bar \rho_\beta \right).
\end{equation}
We claim that
\[
 J_\beta\left(\bar \rho_\beta \right) \le \left( \frac{E(\mf{v}_\beta,0,\bar \rho_\beta) + \frac{N-2}{2} H(\mf{v}_\beta,0,\bar \rho_\beta)}{\bar{\rho}_\beta^2}\right)^2.
\]
To prove it, we firstly test the equation for $v_{1,\beta}$ with $v_{1,\beta} |x|^{2-N}$ in $B_r$; integrating by parts twice we obtain
\begin{align*}
J_{1,\beta}(r) & = -\frac12 \int_{B_r} \nabla(v_{1,\beta}^2)\cdot \nabla (|x|^{2-N}) + \frac{1}{r^{N-2}}\int_{\pa B_r} v_{1,\beta} \pa_{\nu} v_{1,\beta} \\
& \le \frac{1}{r^{N-2}}\int_{\pa B_r} v_{1,\beta} \pa_{\nu} v_{1,\beta} +\frac{N-2}{2r^{N-1}} \int_{\pa B_r} v_{1,\beta}^2.
\end{align*}
Now the divergence theorem yields
\begin{equation}\label{da J a E}
\begin{split}
J_{1,\beta}(r) & \le \frac{1}{r^{N-2}} \int_{B_r} |\nabla v_{1,\beta}|^2 + a_{12} v_{1,\beta}^2 \sum_{j \neq 1} v_{j,\beta}^2 + \frac{N-2}{2r^{N-1}} \int_{\pa B_r} v_{1,\beta}^2 \\
 & \le E(\mf{v}_\beta,0,r) + \frac{N-2}{2}H(\mf{v}_\beta,0,r).
\end{split}
\end{equation}
If we choose $r=\bar \rho_\beta$ and we use the same argument on $J_{2,\beta}$, the claim follows. 

Thus, coming back to \eqref{chain 1} we have
\[
C \le J_\beta(\bar \rho_\beta) \le C \left( \frac{E(\mf{v}_\beta,0,\bar \rho_\beta) + H(\mf{v}_\beta,0,\bar \rho_\beta)}{	\bar \rho_\beta^2} \right)^2,
\]
and on the last term we can apply the monotoncity formula of Proposition \ref{prop: monot e+h}, available firstly in the interval $(\bar \rho_\beta,\bar R_\beta)$ with $\gamma=1$, and secondly in $(\bar R_\beta, \bar r/r_\beta)$ with $\gamma= D-\eps$: recalling \eqref{scaled quantities} and Lemma \ref{lem: N bounded} this gives
\begin{align*}
C & \le \left( \frac{E(\mf{v}_\beta,0,\bar \rho_\beta) + H(\mf{v}_\beta,0,\bar \rho_\beta)}{\bar \rho_\beta^2} \right)^2 \\
& \le \left( \frac{E(\mf{v}_\beta,0,\bar R_\beta) + H(\mf{v}_\beta,0,\bar R_\beta)}{\bar R_\beta^2} \cdot \left( \frac{\bar R_\beta}{\bar R_\beta}\right)^{2(D-\eps-1)}  \right)^2   \\
& \le \left( \frac{E(\mf{v}_\beta,0,\bar r/r_\beta) + H(\mf{v}_\beta,0,\bar r/r_\beta)}{\bar r^{2(D-\eps)}} r_\beta^{2(D-\eps)}  \right)^2 \cdot \bar R_\beta^{4(D-\eps-1)} \\
& = \left( \frac{E(\mf{u}_\beta,0,\bar r) + H(\mf{u}_\beta,0,\bar r)}{\bar r^{2(D-\eps)}} \right)^2 \frac{r_\beta^{4(D-\eps)}}{H(\mf{u}_\beta,x_\beta,r_\beta)^2}  \cdot  \left( \frac{R_\beta}{r_\beta} \right)^{4(D-\eps-1)} \\
& \le C \frac{r_\beta^4 R_\beta^{4(D-\eps-1)} }{H(\mf{u}_\beta,x_\beta,r_\beta)^2},
\end{align*}
whence $H(\mf{u}_\beta,x_\beta,r_\beta)^2 \le C r_\beta^4 R_\beta^{4(D-\eps-1)}$. Finally, using also Lemmas \ref{lem: choice of r} and \ref{lem: H con m}, we deduce that
\begin{align*}
\beta^2\left( \sum_{i=1}^k u_{i,\beta}(x_\beta) \right)^8 & \le C \left(\beta H(\mf{u}_\beta,x_\beta,r)\right)^2 H(\mf{u}_\beta,x_\beta,r)^2 \le C\cdot \frac{1}{r_\beta^4} \cdot  r_\beta^4 R_\beta^{4(D-\eps-1)}
\end{align*}
and since $D>1$ and $R_\beta \to 0$ the last term vanishes as $\beta \to +\infty$, which is the desired result. 
\end{proof}

\begin{remark}
As already pointed out, in order to proof Theorem \ref{thm: global upper estimate} in presence of $f_{i,\beta} \not \equiv 0$ it is possible to combine the techniques used here with the almost monotonicity formulae introduced in \cite{SoZi} (see Theorem 3.14 and Lemma 4.7 therein). 
\end{remark}

\section{Uniform regularity of the interfaces and decay estimates III}\label{sec: Reif}
The aim of this section is to study the uniform regularity of the interfaces $\Gamma_\beta$, and in a second time to prove as a corollary Theorem \ref{corol: decay components vanishing}. 

Before proceeding, we make some remarks about Definition \ref{def: regular interface}, where we introduced $\mathcal{R}_\beta(\rho)$. First, since for $\beta$ finite the functions $\mf{u}_\beta$ are smooth, the function $(x, \rho) \mapsto N_\beta(\mf{u}_\beta, x, \rho)$ is continuous; in particular, for any $\rho>0$, $\mathcal{R}_{\beta}(\rho)$ is a relative open subset of $\Gamma_{\beta}$. In Definition \ref{def: regular interface}, in light of the dichotomy $N(\mf{u},x,0^+) = 1$ or $N(\mf{u},x,0^+) \ge 3/2$ (see Proposition \ref{prop: monot segregated}), we could replace $1/4$ with any positive number strictly less than $1/2$, without affecting the rest of the section. We observe also that, thanks to the monotonicity of the Almgren quotient, for a fixed $\mf{u}_\beta$ we can show the following monotonicity property of the proposed decomposition
\[
	\forall \rho_1, \rho_2, \; 0 < \rho_1 < \rho_2 \implies \mathcal{R}_{\beta}(\rho_1) \supset \mathcal{R}_{\beta}(\rho_2).
\]
The stratification induced by the previous construction on the free boundary $\Gamma_{\beta}$ may seem to be useless: indeed we have
\[
	\Gamma_{\beta} = \cup_{\rho > 0} \mathcal{R}_{\beta}(\rho).
\]
This is due to the fact that the maximum principle implies that all the functions $\mf{u}_\beta$ are strictly positive in $\Omega$, and thus for any $x \in \Omega$ we can easily prove that $N_\beta(\mf{u}_\beta, x, 0^+) = 0$. Nonetheless, the following result can be used to acquire the geometrical intuition behind the definition.
\begin{lemma}\label{lem: basic prop regular part}
Let us assume that $x_\beta \in \Gamma_{\beta}$, for every $\beta$.
\begin{itemize}
	\item If there exists $x_0 \in \mathcal{R}$ such that $x_\beta \to x_0$, then there exist $\rho > 0$ and $\bar \beta >0$ such that
\[
	x_\beta \in \mathcal{R}_{\beta}(\rho) \qquad \forall \beta > \bar \beta.
\]
	\item If there exists $x_0 \in \Sigma$ such that $x_\beta \to x_0$, then for every $\rho > 0$ there exists $\bar \beta >0$ such that
\[
	x_\beta \not \in \mathcal{R}_{\beta}(\rho) \qquad \forall \beta > \bar \beta.
\]
In particular, for any compact $K \Subset \Omega$ and $\rho>0$ there exists $s>0$ independent of $\beta$ such that 
\[
B_s(x) \cap \Sigma = \emptyset \quad \text{for every $x \in \mathcal{R}_\beta(\rho)$, for every $\beta$}.
\] 
\end{itemize}
\end{lemma}
\begin{proof}
We show only the first conclusion, since the second one is similar. Let $x_0 \in \mathcal{R}$; then 
\[
	\begin{split}
		N(\mf{u}, x_0,0^+) = 1 
		&\quad\implies N(\mf{u}, x_0, \rho) < 1 + \frac{1}{2 \cdot 4} \qquad \text{for some small $\rho$}\\
		&\quad\implies N(\mf{u}_\beta, x_\beta, \rho) < 1 + \frac{1}{4} 
	\end{split}
\]
for sufficiently large $\beta$, by the $\mathcal{C}^0_{\loc}(\Omega)$ and the strong $H^1_{\loc}(\Omega)$ convergence of $\mf{u}_\beta$ to $\mf{u}$.
\end{proof}

We now investigate the uniform regularity of the regular part of the subsets $\mathcal{R}_{\beta}(\rho) \cap K$, proving Theorem \ref{thm: reifenberg flat uniform}. Recall that $K$ is an arbitrary compact set in $\Omega$. In order to establish that $\mathcal{R}_{\beta}(\rho) \cap K$ enjoy what we defined as the \emph{uniform vanishing Reifenberg flatness condition}, we proceed in two steps. First of all, we show it under a smallness assumption. 

\begin{lemma}\label{lem: Reif small}
Let $K \Subset \Omega$ be a compact set, $\rho > 0$ and $C>0$. For $\beta$ sufficiently large, for any $\delta > 0$, $x_\beta \in \mathcal{R}_{\beta}(\rho) \cap K$ and $0 < r < C r_\beta(x_\beta)$ there exists a hyper-plane $H_{x_\beta,r} \subset \R^N$ containing $x_\beta$ such that
\[
	\dist_{\mathcal{H}}(\mathcal{R}_{\beta}(\rho) \cap B_r(x_\beta), H_{x_\beta,r} \cap B_r(x_\beta)) \leq \delta r.
\]
\end{lemma}

In the thesis of Theorem \ref{thm: reifenberg flat uniform} we required $R$ to be independent of $\beta$, thus the uniformity of the vanishing Reifenberg flatness of the ``regular part" of the interfaces. Here instead we prove a preliminary result in the case $R = C r_\beta$.

For future convenience, we recall that the notation $B_r$ is used for balls with center in $0$.
\begin{proof}
By contradiction, we suppose that there exist $\bar \delta >0$, $x_\beta \in \mathcal{R}_\beta(\rho) \cap K$ and $0<r_\beta'<C r_\beta$ such that
\[
\inf_H \dist_{\mathcal{H}}(\mathcal{R}_\beta(\rho) \cap B_{r_\beta'}(x_\beta) , H \cap B_{r_\beta'}(x_\beta)) \ge \bar \delta r_\beta' \qquad \forall \beta,
\]
where the infimum is taken over all the hyperplanes passing through $x_\beta$. Since the notion of Reifenberg flatness commutes with translations and scalings, the previous condition is equivalent to 
\begin{equation}\label{contr Reif step 1}
\inf_H \dist_{\mathcal{H}}\left(\mathcal{R}_\beta^{(S)}(\rho) \cap B_{r_\beta'/r_\beta(x_\beta)}  , H \cap B_{r_\beta'/r_\beta(x_\beta)} \right) \ge \bar \delta \frac{r_\beta'}{r_\beta(x_\beta)}
\end{equation}
for every $\beta$, where $\mathcal{R}_\beta^{(S)}(\rho)$ is obtained by $\mathcal{R}_\beta(\rho)$ after the change of variable $x = x_\beta + r_\beta y$, and now the infimum is taken over the hyperplanes through the origin. 

The contradiction will be achieved proving that $\mathcal{R}_\beta^{(S)}(\rho)$ are uniformly Reifenberg flat arround $0$ up to the scale $C$, in the sense that  for any $\delta>0$ and $0<r<C$ it results
\begin{equation}\label{4251}
\inf_H \dist_{\mathcal{H}}(\mathcal{R}_\beta^{(S)}(\rho) \cap B_r  , H \cap B_{r} ) \le \delta r \qquad \forall \beta.
\end{equation}
Since $r_\beta'/r_\beta(x_\beta) \le C$, this contradicts \eqref{contr Reif step 1} and completes the proof. To prove \eqref{4251}, we introduce as usual the sequence 
\[
\mf{v}_\beta(x):= \frac{\mf{u}_\beta(x_\beta + r_\beta x)}{H(\mf{u}_\beta,x_\beta,r_\beta)^{1/2}}.
\]
Since $x_\beta \in \mathcal{R}_\beta(\rho) \cap K$, up to a subsequence $x_\beta \to x_0$. By Proposition \ref{prop: interfaces are good approximation} we have $x_0 \in \Gamma = \mf\{\mf{u}=\mf{0}\}$, and by Lemma \ref{lem: basic prop regular part} it follows that $x_0 \in \mathcal{R}$, the regular part of $\Gamma$. As a consequence, Corollary \ref{thm: classification limits regular part} establishes that $\mf{v}_\beta \to \mf{V}$ in $\mathcal{C}^2_{\loc}(\R^N)$, where $\mf{V}$ is a $1$-dimensional solution of \eqref{entire system a_{ij}}. Up to a rotation and a relabelling, we can suppose that $\{V_1= V_2\} = \{x_N=0\}$ and $V_1, V_2$ are the only nontrivial components of $\mf{V}$. By $\mathcal{C}^2_{\loc}$ convergence, this implies that:
\begin{itemize}
\item $\mathcal{R}_\beta^{(S)}(\rho) \cap B_C = \{v_{1,\beta} -v_{2,\beta}=0\} \cap B_C $;
\item there exists $C_1>0$ such that $|\pa_{x_N} (v_{1,\beta}-v_{2,\beta})| > C_1>0$ in $B_C$, for every $\beta$;
\item for every $\delta>0$ there exists $\bar \beta>0$ such that $|\pa_{x_i} (v_{1,\beta}-v_{2,\beta})| < \delta/(C_1(N-1))$ in $B_C$ provided $\beta >\beta$.
\end{itemize}
Therefore, for $\beta>\bar \beta$ we can apply the implicit function theorem: there exists a $\mathcal{C}^1$ function $f_\beta$, defined on the projection $U_\beta$ of $\mathcal{R}_\beta^{(S)}(\rho) \cap B_C $ into $\R^{N-1}$, such that $\mathcal{R}_\beta^{(S)}(\rho) \cap B_C  = \{x_N= f_\beta(x')\}$. Moreover, $f_\beta(0)  = 0$ (since $0 \in \mathcal{R}_\beta^{(S)}(\rho) \cap B_C $) and $|\nabla' f_\beta| \le \delta$ in $U_\beta$. As a result, choosing $\bar H= \{x_N=0\}$, and denoting by $U_\beta^r$ the set $U_\beta \cap \{|x'|<r\}$, we have
\begin{align*}
	\dist_{\mathcal{H}}(\mathcal{R}_{\beta}^{(S)}(\rho) \cap B_r , \bar H \cap B_r )  &	= \sup_{\mathcal{R}_{\beta}^{(S)}(\rho) \cap B_r } |x_N| \le \sup_{U_\beta^r} |f_\beta| \\
	& \le  \sup_{U_\beta^r} | \nabla' f_\beta| |x'| \leq  \delta r,
\end{align*}
which gives the desired contradiction.
\end{proof}

\begin{proof}[Proof of Theorem \ref{thm: reifenberg flat uniform}]
We now conclude the proof of the uniform vanishing Reifenberg flatness of the sets $\mathcal{R}_\beta(\rho)$. 
By contradiction again, let us assume that there exist $\bar \delta >0$ and sequences $\beta_n \to +\infty$, $x_{n} \in \mathcal{R}_{\beta_n}(\rho) \cap K$, $r_{n} \to 0^+$ such that
\begin{equation}\label{eqn: failure reif}
	\dist_{\mathcal{H}}(\mathcal{R}_{\beta_n}(\rho) \cap B_{r_n}(x_{n}), H \cap B_{r_n}(x_{n})) \geq \bar \delta r_n
\end{equation}
for every $H$ hyperplane passing through $x_n$. We start by the simple observation that, thanks to Lemma \ref{lem: Reif small}, a constant $C >0$ such that $r_n < C r_{\beta_n}(x_{n})$ cannot exist: in other words, it must be
\begin{equation}\label{eqn: rn to infinity}
	\liminf_{n \to \infty} \frac{r_n}{r_{\beta_n}(x_{n})} = +\infty.
\end{equation}
Now we introduce the scaled functions
\[
	\mf{w}_n(x) = \frac{1}{\sqrt{H(\mf{u}_{\beta_n}, x_n, r_n)}} \mf{u}_{\beta_n}(x_n + r_n x).
\]
The equation for $\mf{w}_n$ is 
\[
\Delta w_{i,n} = r_n^2 H(\mf{u}_{\beta_n},x_n,r_n) \beta_n w_{i,n} \sum_{j \neq i} a_{ij} w_{j,n}^2,
\] 
and by \eqref{eqn: rn to infinity} and the choice of $r_{\beta_n}(x_n)$, Lemma \ref{lem: choice of r}, the interaction parameter is
\[
r_n^2 H(\mf{u}_{\beta_n},x_n,r_n) \beta_n = r_{\beta_n}^2 H(\mf{u}_{\beta_n},x_n,r_n) \beta_n \cdot \left(\frac{r_n}{r_{\beta_n}(x_n)}\right)^2 \to +\infty.
\]
Moreover, for any $R>1$ and $0<r<R$
\[
N(\mf{w}_n,0,r) \le N(\mf{w}_n,0,R) = N(\mf{u}_{\beta_n},x_n,r_n R) \le  N(\mf{u}_{\beta_n},x_n,\rho) \le \frac{5}{4}
\]
provided $n$ is sufficiently large, which implies
\[
\frac{d}{dr} \log H(\mf{w}_n,0,r) \le \frac{5}{2r} \quad \Longrightarrow \quad H(\mf{w}_n,0,R) = \frac{H(\mf{w}_n,0,R)}{H(\mf{w}_n,0,1)} \le R^{5/2}. 
\]
In turn, by subharmonicity, and since $R$ has been arbitrarily chosen, this ensures that $\{\mf{w}_n\}$ is locally bounded in $L^\infty$, and applying as usual \cite{SoTaTeZi} (see also \cite{NoTaTeVe,tt,Wa}) we finally infer that $\mf{w}_n \to \mf{W}  \in \mathcal{G}_{\loc}(\R^N)$, locally uniformly and in $H^1_{\loc}(\R^N)$. We recall that the main properties of the class $\mathcal{G}$ have been reviewed in Section \ref{sec: prel}, and we point out that $\mf{W} \not \equiv \mf{0}$ since the $L^2$-norm of $\mf{W}$ on the unit sphere is normalized to $1$. Directly from the convergence we deduce that $N(\mf{W},0,r) \le 5/4$ for every $r>0$. Actually a stronger estimate holds, since for any $r,\tilde r>0$ we have
\begin{align*}
N(\mf{W},0,r) &= \lim_{n \to \infty} N(\mf{w}_n,0,r)  = \lim_{n \to \infty} N(\mf{u}_{\beta_n},x_n,r_n r)\\
& \le \lim_{n \to \infty} N(\mf{u}_{\beta_n},x_n,\tilde r) = N(\mf{u},x_0,\tilde r),   
\end{align*}
where we used the compactness of $K$ to infer that $x_n \to x_0$. Notice that, by Lemma \ref{lem: basic prop regular part}, $x_0 \in \mathcal{R}$. Therefore, since $r$ and $\tilde r$ in the previous estimate are arbitrarily chosen, we can pass to the limit as $r \to +\infty$ and $\tilde r \to 0^+$, deducing that $N(\mf{W},0,+\infty) \le 1$. Using also the monotonicity of the Almgren quotient and the lower bound on $N(\mf{W},0,0^+)$ (see Proposition \ref{prop: monot segregated}), we conclude that 
\[
1 \le N(\mf{W},0,0^+) \le N(\mf{W},0,+\infty) \leq 1 \quad \Longrightarrow \quad N(\mf{W},0,r) = 1 \quad \forall r.
\]
As a consequence, up to a rotation and a relabelling $\mf{W} = \alpha(x_N^+, x_N^-, 0, \dots, 0)$ for some positive $\alpha$, and in particular $\{\mf{W}=\mf{0}\} = \{x_N = 0\}$. 

To complete the proof, we observe that scaling \eqref{eqn: failure reif} we have
\[
	\dist_{\mathcal{H}}(\mathcal{R}^{(S)}_{\beta_n}(\rho) \cap B_{1} , H \cap B_{1} ) \geq \bar \delta \quad \text{for every hyperplane $H$ passing in $0$},
\]
for every $\beta$. On the other hand, by the uniform convergence $\mf{w}_n \to \mf{W}$ it is not difficult to check that 
\begin{equation}\label{4821}
\dist_{\mathcal{H}}(\mathcal{R}^{(S)}_{\beta_n}(\rho) \cap B_{1} , \{x_N=0\} \cap B_{1} ) \to 0 \qquad \text{as $n \to +\infty$},
\end{equation}
which gives the sought contradiction (concerning the detailed verification of \eqref{4821}, we refer the interested reader to the proof of Lemma 5.3 in \cite{tt}, where the authors deal with a similar context).
\end{proof}

An important consequence of the Reifenberg flatness of the free boundary is given by a local separation property. We write that a set $\omega \subset \Omega$ \emph{separates $\Omega$ in a neighbourhood of $x \in \omega$} if there exists $r>0$ such that $B_r(x) \subset \Omega$ and $B_r(x) \setminus \omega$ consists of two connected components. As we shall see, the interface $\Gamma_\beta$ enjoys this important property in a neighbourhood of any point $x \in \mathcal{R}_\beta(\rho)$, with separation radius uniform in $x$. Consequently, we have that:
\begin{itemize}
	\item in a $R$-neighbourhood of $\mathcal{R}_\beta(\rho) \cap K$ (with $R$ independent of $\beta$), the interface $\Gamma_\beta$ never self-intersects;
	\item in a $R$-neighbourhood of $\mathcal{R}_\beta(\rho) \cap K$ (with $R$ independent of $\beta$), two densities dominate on the other ones.
\end{itemize}

\begin{proposition}\label{prp: local sep}
Let $K \Subset \Omega$ be a compact set and let $\rho > 0$. There exists $R > 0$ such that  $B_R(x_\beta) \cap \Gamma_\beta$ has exactly two connected components for every $x_\beta \in \mathcal{R}_\beta(\rho)$.
\end{proposition}
The proof of this result is very similar to the one given in the limit setting by Tavares and Terracini in \cite{tt}, which was in turn based on the \cite[Theorem 4.1]{HongWang}. Thus, we only sketch it.

\begin{proof}
The fundamental observation here is that the family $\mathcal{R}_\beta(\rho)$ consists of sets which enjoy the uniform vanishing Reifenberg flatness property: as a consequence, if one proves that the local separation property holds for one of them, and the proof is based only on uniform-in-$\beta$ assumptions, the general case follows immediately. 

Let $\rho > 0$ be fixed, we consider a small $\delta$-flatness parameter ($\delta < 1/6$ for instance is sufficient), and let $R' = R'(\delta)$ the uniform-in-$\beta$ radius for which the $(\delta,R')$-Reifenberg flatness condition holds for each set $\mathcal{R}_\beta(\rho)$. Let also $s>0$ be defined by Lemma \ref{lem: basic prop regular part}. We define $R:= \min\{s/2,R/2\}$, and we show that this is a local separation radius for every $x \in \mathcal{R}_\beta(\rho)$, for every $\beta$. To this aim, we can replicate almost word by word the proof of \cite[Proposition 5.4]{tt} 
In particular, since $B_R(x) \cap \mathcal{R}_\beta(\rho)$ is $(\delta,R)$-Reifenberg flat and is detached from $\Sigma$, the set $B_R(x) \cap \mathcal{R}_\beta(\rho)$ is trapped between to parallel hyperplanes at distance $2\delta$, and the complementary region is given by two open and disjoint subsets of $B_R(x)$. We now consider inductively the radius $R/2^k$, $k \geq 1$ and balls $B_{R/2^k}(y)$ centered at points $y \in B_R(x) \cap \mathcal{R}_\beta(\rho)$ and the new connected components generated by the respective trapping hyperplanes. Thanks to the fact that $\delta$ is small, it is possible to show that each of these pairs of new components intersect one and only one of the connected components of the previous step. Joining all the corresponding sets we find two new connected components of $B_R(x)$ that are at distance $\delta/2^{k-1}$, and set $B_R(x) \cap \mathcal{R}_\beta(\rho)$ is again trapped between the two. Iterating this process we conclude the proof.
\end{proof}

Using the properties so far shown for $\mathcal{R}_\beta(\rho)$, we can better describe the behaviour of the functions near the interface set.
\begin{proposition}\label{clean up regular}
Let $K \Subset \Omega$ be a compact set, $\rho > 0$, and let $R > 0$ be the separation radius of Proposition \ref{prp: local sep}, independent of $\beta$. For any $x \in \mathcal{R}_\beta(\rho) \cap K$, there exist two indices $i_1 \neq i_2$ such that:
\begin{itemize}
	\item $\mathcal{R}_\beta(\rho) \cap B_R(x) = \{ u_{i_1, \beta} = u_{i_2, \beta} \} \cap B_R(x)$ and moreover the two connected components of $  B_R(x) \setminus \mathcal{R}_\beta(\rho)$ are given by $\{ u_{i_1, \beta} > u_{i_2, \beta} \} \cap B_R(x)$ and $\{ u_{i_1, \beta} < u_{i_2, \beta} \} \cap B_R(x)$;
	\item for any $j \neq i_1, i_2$, the density $u_{j,\beta}$ is exponentially small with respect to $u_{i_1,\beta}$ and $u_{i_2,\beta}$, in the sense that there exist $C_1,C_2>0$ such that 
	\[
\qquad		\sup_{ B_{R/2}(x)} u_{j,\beta} \leq C_1 e^{- C_1 \beta^{C_2}};
	\]
	\item in $B_{R/2}(x)$ the system reduces to 
	\[
		\begin{cases}
			-\Delta u_{i_1,\beta} = - \beta u_{i_1,\beta} u_{i_2,\beta}^2 - u_{i_1,\beta}o_\beta(1) \\
			-\Delta u_{i_2,\beta} = - \beta u_{i_2,\beta} u_{i_1,\beta}^2 - u_{i_2,\beta}o_\beta(1)
		\end{cases}
	\]
	where $o_\beta(1)$ is a (exponentially) small perturbation in the $L^{\infty}$-norm.
\end{itemize}
\end{proposition}
As we shall see, Theorem \ref{corol: decay components vanishing} is a simple consequence of this proposition together with the compactness of $K$ and the definition of $\mathcal{R}_\beta(\rho)$.
\begin{proof}
For $x \in \mathcal{R}_\beta(\rho) \cap K$, the set $B_R(x) \setminus \Gamma_\beta$ is given by two connected components. By Lemma \ref{lem: basic prop regular part} and by the choice of the local separation radius $R \le s/2$, it follows also that $B_R(x) \cap \Sigma_\beta =\emptyset$, where we recall that the singular part of the interface was introduced in Definition \ref{def singular interface}. Indeed, if this is not the case we can find a sequence $x_\beta \in K \cap \mathcal{R}_\beta(\rho)$ and, correspondingly, $y_\beta \in B_R(x_\beta) \cap \Sigma_\beta$. By compactness and Corollary \ref{cor: improved decay singular sequence}, we deduce that $y_\beta \to y \in \Sigma$, in contradiction with the second point in Lemma \ref{lem: basic prop regular part}  and the fact that $R \le s/2$. Therefore, in each of the connected components of $B_R(x) \setminus \Gamma_\beta$, one function dominates the others $k-1$, and by \cite[Section 10]{DaWaZh} the two dominating functions must be different. 
We explicitly remark that, if necessary replacing $R$ with a smaller quantity, it is possible to assume that 
\[
\text{the closure of }\left(\bigcup_{\beta} \bigcup_{x \in \mathcal{R}_\beta(\rho) \cap K} B_{R}(x)\right) \quad \text{is a compact subset of $\Omega$}.
\] 
Therefore, by Lemma \ref{lem: N bounded}, there exists $\bar C>0$ independent of $\beta$ such that
\begin{equation}\label{universal bound on N}
\sup_{\beta} \, \sup_{x \in \mathcal{R}_\beta(\rho) \cap K} \, \sup_{y \in B_R(x)} N(\mf{u}_\beta,y,R/4) \le \bar C.
\end{equation}

Let $\tilde C:= 1/(2+2 \bar C)$. We claim that there exists $C>0$ such that
\begin{equation}\label{eqn lower clean up}
\inf_{x \in \mathcal{R}_\beta(\rho) \cap K}	\inf_{y \in B_{3R/4}(x)} \sum_{i = 1}^{k} u_{i,\beta}(y) \geq C \beta^{-\frac12 + \tilde C}.
\end{equation}
To prove the previous claim, we argue as in Theorem \ref{thm: lowe estimate}. Suppose by contradiction that the claim is not true: then there exist sequences $\beta \to +\infty$, $x_\beta \in \mathcal{R}_\beta(\rho)$ and $y_\beta \in B_{3R/4}(x_\beta)$ such that
\begin{equation}\label{abs eqn lower bound clean up}
\lim_{\beta \to +\infty} \beta^{\frac{1}{2}-\tilde C}\sum_{i =1}^k u_{i,\beta}(y_\beta) =0.
\end{equation}

Thus $y_\beta \to \bar y \in \Omega$, and since \eqref{nontrivial} is in force, we find a sequence $r_\beta = r_\beta(y_\beta) \to 0$ as in Lemma \ref{lem: choice of r}. Moreover, recalling \eqref{universal bound on N} and \eqref{prop: almgren}, we have also 
\[
\frac{d}{dr}\log H(\mf{u}_\beta,y_\beta,r) \le \frac{2 \bar C}{r} \qquad \forall 0 < r < \frac{R}{4},
\] 
whence by \eqref{nontrivial} we infer
\[
\frac{H(\mf{u}_\beta,y_\beta,r_\beta)}{r_\beta^{2 \bar C}} \ge \frac{H(\mf{u}_\beta,y_\beta,R/4)}{R^{2 \bar C}} \ge C.
\] 
This estimate can be used as in Theorem \ref{thm: lowe estimate}: thanks to Lemmas \ref{lem: choice of r} and \ref{lem: H con m},
\begin{align*}
\left(\sum_{i=1}^k u_{i,\beta}(y_\beta)\right)^2 & \ge C 
H(\mf{u}_\beta,y_\beta,r_\beta)  \ge C  {r_\beta^{2 \bar C}} \ge \frac{C}{H(\mf{u}_\beta,y_\beta,r_\beta)^{\bar C} \beta^{\bar C} } \\
& \ge \frac{C}{\beta^{\bar C}} \left(\sum_{i=1}^k u_{i,\beta}(y_\beta)\right)^{-2 \bar C}.
\end{align*}
It is not difficult to obtain a contradiction with \eqref{abs eqn lower bound clean up}, thus proving claim \eqref{eqn lower clean up}. 

By the local separation property we know that for any $x \in \mathcal{R}_\beta(\rho)$ there are two indices $i_1,i_2$ such that the functions $u_{i_1,\beta}$ and $u_{i_2,\beta}$ are dominating the remaining $k-2$ components in $B_{3R/4}(x)$. Combining this with \eqref{eqn lower clean up}, we obtain
\[
	\inf_{y \in B_{3R/4}(x)} \left( u_{i_1,\beta}(y) + u_{i_2,\beta}(y) \right) \geq C \beta^{-\frac12 + \tilde C}
\]
(here $x$ depends on $\beta$, and $i_1,i_2$ could depend both on $x$ and on $\beta$, but we do not stress this to keep the notation simple; what it is important is that $R$ is independent of $\beta$). To complete the proof, we shall use the previous estimate in the equation satisfied by the function $u_{j,\beta}$, $j \neq i_1, i_2$ in the ball $B_{3R/4}(x)$: this gives
\[
	- \Delta u_{j,\beta} = - \beta u_{j,\beta} \sum_{i \neq j} u_{i,\beta}^2 \leq - C \beta u_{j,\beta} \left( u_{i_1,\beta} + u_{i_2,\beta} \right)^2 \leq - C \beta^{2\tilde C} u_{j,\beta},
\]
and thus, invoking Lemma \ref{lem: decay} and assumption \eqref{boundedness}, we finally infer
\[
	\sup_{B_{R/2}(x)} u_{j,\beta} \leq C e^{- C \beta^{2\tilde C}},
\]
proving the second point in the thesis. The third point follows easily.
\end{proof}

Theorem \ref{corol: decay components vanishing} is a simple corollary of the previous statement.  
\begin{proof}[Proof of Theorem \ref{corol: decay components vanishing}]
Under the assumptions of the corollary, there exists $x_\beta \in \Gamma_\beta$ such that $x_\beta \to x_0$, see Proposition \ref{prop: interfaces are good approximation}. Moreover, $x_\beta \not \in \Sigma_\beta$, otherwise we would have a contradiction with Corollary \ref{thm: non-simple blow-up}. We claim that there exists $\rho>0$ (independent of $\beta$) such that $x_\beta \in \mathcal{R}_\beta(\rho) \cap K$ for every $\beta$. Once that this is proved, the thesis follows by Proposition \ref{clean up regular}. Suppose by contradiction that a value $\rho$ as before does not exist. Then there exists $\rho_\beta \to 0^+$ such that
\[
N(\mf{u}_\beta,x_\beta,\rho_\beta) \ge 1+ \frac{1}{4}.
\]
On the other hand, since $x_0 \in \mathcal{R}$ there exists $\bar r>0$ such that $N(\mf{u},x_0,\bar r) \le 1+1/8$, and by monotonicity of the Almgren quotient and the usual convergence we easily reach a contradiction:
\[
1+\frac{1}{8}  >N(\mf{u},x_0,\bar r)  = \lim_{\beta \to +\infty} 
N(\mf{u}_\beta,x_\beta,\bar r) 
 \ge \lim_{\beta \to +\infty}  N(\mf{u}_\beta,x_\beta,\rho_\beta) \ge 1+ \frac{1}{4}.
\]
This proves the existence of $\rho$, and in turn the desired result.
\end{proof}

We conclude this section with the:

\begin{proof}[Proof of Proposition \ref{prop: non C^1}]
We can provide a counterexample to the convergence of the gradients. As reviewed in the preliminaries, there exists a unique solution to the system of ordinary differential equations
\[
\begin{cases}
u'' = uv^2 \\
v'' = u^2 v \\
u,v>0
\end{cases} \quad \text{in $\R$}, \quad \text{with $u'(+\infty) = 1$ and $v(x) = u(-x)$}.
\]
Notice that, consequently, for the (constant) Hamiltonian function we have 
\[
(u')^2(x) + (v')^2(x) -u^2(x)v^2(x) =1 \qquad \forall x \in \R.
\]
Let us consider 
\[
(u_R(x),v_R(x)):= \frac{1}{R}(u(Rx), v(Rx)).
\]
This is a sequence of solutions to \eqref{system simplified} with $\beta(R)=R^4 \to +\infty$, and it is not difficult to deduce by usual arguments that it is locally uniformly bounded in $L^\infty$. Thus, by \cite{SoTaTeZi} (see also \cite{NoTaTeVe,Wa}), it is convergent in $\mathcal{C}^0_{\loc}(\R)$ and in $H^1_{\loc}(\R)$, up to a subsequence, to a limit profile $(U,V)$, such that $U-V$ is harmonic, and thus affine, in $\R$. Since $u'_R(1) \to 1$ as $R \to +\infty$, and since $u_R \to U$ in $\mathcal{C}^1_{\loc}(\R \setminus\{0\})$, we deduce that $(U,V)=(x^+,x^-)$. Let us suppose now by contradiction that $u_R-v_R \to U-V$ in $\mathcal{C}^1([-\eps,\eps])$ for some $\eps>0$; then, recalling the symmetry of the solution, we infer that
\[
1= U'(0)-V'(0) = \lim_{R \to \infty} u_R'(0) - v_R'(0) = u'(0) -v'(0) = 2 u'(0),
\]
so that $u'(0) =-v'(0) = 1/2$. Coming back to the definition of the energy, we finally obtain
\[
1= \frac{1}{2} - u^2(0) v^2(0) < 1,
\]
a contradiction.
\end{proof}

\noindent \textbf{Acknowledgements:} part of this work was carried out while Nicola Soave was visiting the Centre d'Analyse et de Math\'{e}matique Sociales in Paris, and he wishes to thank for the hospitality. The authors are partially supported through the project ERC Advanced Grant 2013 n. 339958 ``Complex Patterns for Strongly Interacting Dynamical Systems - COMPAT''. Alessandro Zilio is also partially supported by the ERC Advanced Grant 2013 n. 321186 ``ReaDi -- Reaction-Diffusion Equations, Propagation and Modelling''.


\end{document}